\documentclass[11pt]{amsart}

\usepackage{amscd,amsmath,amssymb,amsfonts,amsthm,mathrsfs,mathtools}
\usepackage[a4paper,text={160mm,240mm},centering,headsep=5mm,footskip=10mm]{geometry}
\usepackage[all]{xy}

\usepackage[colorlinks,linkcolor=black,anchorcolor=black,citecolor=black]{hyperref}

\usepackage[square,sort,comma,numbers]{natbib}
\usepackage{stmaryrd}
\usepackage{array}
\usepackage{color}

\newtheorem{theorem}{Theorem}[subsection]
\newtheorem{proposition}[theorem]{Proposition}

\newtheorem{remark}[theorem]{Remark}
\newtheorem{corollary}[theorem]{Corollary}

\newtheorem{definition}[theorem]{Definition}

\newtheorem{claim}[theorem]{Claim}

\numberwithin{equation}{subsection}

\newcommand{\gr}{\mathrm{gr}}

\newcommand{\yigengrem}[1]{\begin{color}{red}{#1}\end{color}}

\def\ul#1{\underline{#1}}

\def\DRW#1#2#3{W_{#1}\Omega^{#2}_{#3}}
\def\DRWlog#1#2#3{W_{#1}\Omega^{#2}_{#3,\log}}

\def\DRWlognc#1#2#3{\nu_{#1,#3}^{#2}}

\def\Spec{\mathrm{Spec}}
\def\Hom{\mathrm{Hom}}
\def\Supp{\mathrm{Supp}}

\def\cO{\mathcal{O}}

\def\bN{\mathbb{N}}

\def\Fp{\mathbb{F}_p}
\def\Fq{\mathbb{F}_q}
\def\Z{\mathbb{Z}}
\def\Zp1{\mathbb{Z}/p\mathbb{Z}}
\def\Zpm{\Z/p^m\Z}
\def\Zlm{\Z/\ell^m\Z}

\def\D#1{D_{\underline{#1}}}
\def\nlam{(n_{\lambda})_{\lambda\in \Lambda}}

\def\gr{\text{gr}}
\def\Fil{\mathrm{Fil}}
\def\cHom{\mathscr{H}\mathrm{om}}

\def\cRF{W_m\Omega^d_{X|D,\log}}
\def\cAF{W_m\Omega^d_{X,\log}}
\def\et{\text{\'et}}
\def\nis{\text{Nis}}
\def\cF{\mathcal{F}}
\def\m{\mathfrak{m}}
\def\cRM#1{\mathcal{K}^M_{#1, X|D}}
\def\cRFr#1{W_m\Omega^{#1}_{X|D,\log}}

\def\pinP{P\in \mathcal{P}}

\begin{document}

\title[higher ideles and class field theory]{higher ideles and class field theory}

\author{Moritz Kerz}
\address{Fakult\"at f\"ur Mathematik, Universit\"at Regensburg, 93040 Regensburg, Germany}
\email{moritz.kerz@mathematik.uni-regensburg.de}

\author{Yigeng Zhao}
\address{Fakult\"at f\"ur Mathematik, Universit\"at Regensburg, 93040 Regensburg, Germany}
\email{yigeng.zhao@mathematik.uni-regensburg.de}

\thanks{The authors are supported by the DFG through CRC 1085 \emph{Higher Invariants} (Universit\"at Regensburg)}

\begin{abstract}
	We use higher ideles and duality theorems to develop a universal approach to higher dimensional class field theory.
\begin{flushright}
 {\emph{Dedicated to Professor Shuji Saito\\
 on the occasion of his 60th birthday}}
\end{flushright}
\end{abstract}

\maketitle

\tableofcontents

\setcounter{page}{1}
\setcounter{section}{0}
\pagenumbering{arabic}
\section*{Introduction}

In higher dimensional class field theory one tries to describe the abelian fundamental
group of a scheme $X$ of arithmetic interest in terms of idelic or cycle theoretic data on
$X$. More precisely, assume that $X$ is regular and connected and fix a modulus data,
i.e.~an effective divisor $D$ on $X$. Let $\pi_1^{\rm ab}(X,D)$ be the abelian fundamental
group classifying \'etale coverings with ramification bounded by $D$. One defines an idele
class group $C(X,D)$ which is a quotient of the idele group
\[
I(U\subset X)\coloneqq \bigoplus_{\pinP}K^M_{d(P)}(k(P))
\]
by a modulus subgroup depending on $D$ and certain reciprocity relations. Here $\pinP$
runs through some set of chains of prime ideals and $k(P)$ is a generalized form of
henselian local
residue field at the chain $P$, see Subsection~\ref{subsec:higherideles} and~\cite{kerzideles}.

One then constructs a residue map
\[
\rho: C(X,D) \to \pi_1^{\rm ab}(X,D)
\]
which we show to be an isomorphism after tensoring with $\mathbb Z/n\mathbb Z$ ($n>0$) in
the following situations:

\begin{itemize}
	\item[(i)] $X$ is a smooth proper variety over a finite field, recovering (with simpler
	proof) the main result of~\cite{katosaito} for varieties, see Section~\ref{sec:cftvarieties}.
	\item[(ii)] $X$ is an (equal characteristic) complete regular local ring with finite residue field, recovering in case
	$\dim(X)=2$ results of~\cite{saitoCFT2dim}, recovering in case $n$ is invertible on
	$X$ results of~\cite{satoladiccft} and completing our understanding in case $X$ is of
	equal characteristic $p$ and $n$ is a power of $p$, see
	Section~\ref{sec:cftlocalrings}.
	\item[(iii)] $X$ is a smooth proper scheme over an (equal characteristic) complete discrete valuation ring with
	finite residue field, recovering results of Bloch and Saito, see~\cite{saito85cft}, for $\dim(X)=2$ and results
	of~\cite{forre15} for $n$ invertible on $X$ and completing our understanding in case $X$
	is of characteristic $p$ and $n$ is a power of $p$, see Section~\ref{sec:cftschemeslocal}.
\end{itemize}

Here is an outline of our universal strategy to all three cases of the reciprocity
isomorphism $\rho$ in higher dimensional class
field theory listed above:

\smallskip

{\em Step 1}: Show that $C(X,D)$ is isomorphic to a Nisnevich cohomology group of relative
Milnor $K$-sheaf $\mathcal K^M_{X,D}$, for example in case (i) above one has an
isomorphism
\[
C(X,D) \cong H^d(X_{Nis}, \mathcal K^M_{d,X|D}),
\]
where $d=\dim(X)$.

\smallskip

{\em Step 2}: Show that the Nisnevich cohomology of the relative Milnor $K$-sheaf with
finite coefficients is
isomorphic to a certain analogous \'etale cohomology group, for example in case (i) and
for $n=p^m$ a power of the characteristic $p$ of the base field one has an isomorphism
\[
H^d(X_{\nis}, \mathcal K^M_{d,X|D}/n) \cong H^d(X_{\et} , W_m\Omega^d_{X|D,\log} )
\]
where $ W_m\Omega^d_{X|D,\log}$ is a relative de Rham-Witt sheaf. This isomorphism is
established by comparing coniveau spectral sequences and observing that based on
cohomological dimension arguments there is just one
additional potentially non-vanishing row in the spectral sequence in the \'etale situation, which however
disappears at the end by known cases of the Kato conjecture.

\smallskip

{\em Step 3}: Arithmetic duality tells us that the \'etale cohomology group from Step~2 is
isomorphic to an abelian \'etale fundamental group, for example in the special case as in
Step~2 the pro-finite group
$
\lim_D H^d(X_{\et} , W_m\Omega^d_{X|D,\log} )
$,
where $D$ runs through all effective divisors with a fixed support $X\setminus U$, is
Pontryagin dual to the (discrete) cohomology group $H^1(U_\et , \mathbb Z/n\mathbb Z)$.

\section{Higher ideles and Milnor $K$-sheaves}
\subsection{Higher ideles}\label{subsec:higherideles}
Let $X$ be an integral noetherian scheme with a dimension function $d$. Recall that a dimension function on a scheme $X$ is a set theoretic function $d\colon X \to \Z$ such that
\begin{itemize}
\item[(i)] for all $x\in X$, $d(x)\geq 0$;
\item[(ii)] for $x, y \in X$ with $y \in \overline{\{x\}}$ of codimension one, $d(x)=d(y)+1$, where $\overline{\{x\}}$ denotes the closure of $\{x\}$ in $X$.
\end{itemize}

We also denote $d=d(\eta)$, where $\eta$ is the generic point of $X$. Let $d_m$ be the minimal of the integers $d(x)$ for $x\in X$. For an effective Weil divisor $D$ of $X$, we denote $U=X\setminus D$.

\begin{definition}
\begin{itemize}
\item[(i)] A chain on $X$ is a sequence of points $P=(p_0,p_1,\cdots,p_s)$ of $X$ such that
\[ \overline{\{p_0\}} \subset  \overline{\{p_1\}} \subset \cdots \subset  \overline{\{p_s\}};\]
\item[(ii)] A Parshin chain on $X$ is a chain $P=(p_0,p_1,\cdots,p_s)$ on $X$ such that $d(p_i)=i+d_m$, for $0\leq i \leq s$;
\item[(iii)] A Parshin chain on the pair $(U\subset X)$ is a Parshin chain $P=(p_0,p_1,\cdots,p_s)$ on $X$ such that $p_i\in D$ for $0\leq i<s$ and such that $p_s\in U$.
\item[(iv)] The dimension $d(P)$ of a chain $P=(p_0,p_1,\cdots,p_s)$ is defined to be $d(p_s)$;
\item[(v)] A Q-chain on $(U\subset X)$ is defined as a chain $P=(p_0,\cdots,p_{s-2}, p_s)$ on $X$ for $1\leq s \leq d$, such that $d(p_i)=i+d_m$ for $i\in \{ 0,1,\cdots, s-2,s\}$, $p_i\in D$ for $0\leq i \leq s-2$ and $p_s\in U$.
\end{itemize}
\end{definition}
 We also recall the definition of Milnor $K$-theory.
 \begin{definition}
 \begin{itemize}
 \item[(i)] For a commutative unital ring $R$, the Milnor $K$-ring $K^M_{\bullet}(R)$ of $R$ is the graded ring $T(R^{\times})/I$, where $I$ is the ideal of the tensor algebra $T(R^\times)$ over $R^{\times}$ generated by elements $a\otimes (1-a)$ with $a,1-a\in R^\times$. The image of $a_1\otimes \cdots \otimes a_r$ in $K^M_r(R)$ is denoted by $\{a_1,\cdots,a_r\}$.
 \item[(ii)] If $R$ is a discrete valuation ring with quotient field $K$ and maximal ideal $\m \subset R$ we define $K^M_r(K,n)\subset K^M_r(K)$ be the subgroup generated by $\{1+\m^n, K^\times,\cdots,K^\times \}$ for an integer $n\geq 0$.
 \end{itemize}
 \end{definition}
\begin{definition}Let $P=(p_0,\cdots, p_s)$ be a chain on $X$.
\begin{itemize}
\item[(i)] We define the ring $\cO_{X,P}^h$, which is a finite product of henselian local rings, as follows: If $s=0$ set $\cO_{X,P}^h=\cO_{X,p_0}^h$. If $s>0$ assume that $\cO_{X,P'}^h$ has been defined for chains of the form $P'=(p_0,\cdots, p_{s-1})$. Denote $R=\cO_{X,P'}^h$, let  $T$ be the finite set of prime ideals of $R$ lying over $p_s$. Then we define
\[ \cO_{X,P}^h\coloneqq\prod_{\mathfrak{p}\in T}R_{\mathfrak{p}}^h;\]
\item[(ii)] For a chain $P=(p_0,\cdots, p_s)$ on $X$ we let $k(P)$ be the finite product of the residue fields of  $\cO_{X,P}^h$. If $s\geq 1$ each of these residue fields has a natural discrete valuation such that the product of their rings of integers is equal to the normalization of $\cO_{X,P'}^h/p_s$, where $P'=(p_0,\cdots,p_{s-1})$.
\end{itemize}
\end{definition}

Let $\mathcal{P}$ be the set of Parshin chains on the pair $(U\subset X)$, and let $\mathcal{Q}$ be the set of $Q$-chains on $(U\subset X)$ . For a Parshin chain $P=(p_0,\cdots, p_{d-d_m})\in \mathcal{P}$ of dimension $d$ we denote $D(P)$ the multiplicity of the prime divisor $\overline{\{p_{d-d_m-1}\}}$ in $D$.
\begin{definition}
\begin{itemize}
\item[(i)] The idele class group of $(U\subset X)$ is defined as
\[ I(U\subset X)\coloneqq \bigoplus_{\pinP}K^M_{d(P)}(k(P)), \]
and endow this group with the topology generated by the open subgroups
\[ \bigoplus_{\substack{\pinP\\d(P)=d}}K^M_d(k(P), D(P)) \subset I(U\subset X), \]
where $D$ runs through all effective Weil divisors with support $X\setminus U$;
\item[(ii)] The idele group of $X$ relative to the fixed effective divisor $D$ with complement $U$ is defined as
\[ I(X,D)\coloneqq\mathrm{Coker}(\bigoplus_{\substack{\pinP\\d(P)=d}}K^M_d(k(P), D(P)) \to I(U\subset X)); \]
\item[(iii)]The idele class group $C(U\subset X)$ is
\[ C(U\subset X)\coloneqq\mathrm{Coker}(\bigoplus_{P\in \mathcal{Q}}K^M_{d(P)}(k(P))\xrightarrow{Q} I(U\subset X)), \]
where $Q$ is defined to be the sum of all $Q^{P'\to P}$ for $P'=(p_0,\cdots,p_{s-2},p)\in \mathcal{Q}$ and $P=(p_0,\cdots, p_{s-2},p_{s-1},p_s) \in \mathcal{P}$:
\begin{itemize}
\item if $p_{s-1}\in D$, then $Q^{P'\to P}$ is the natural map $K^M_{d(P')}(k(P'))\to K^M_{d(P)}(k(P)) $ induced on Milnor $K$-groups by the ring homomorphism $k(P')\to k(P)$;
\item if $p_{s-1}\in U$, then $Q^{P'\to P}$ is the residue symbol $K^M_{d(P')}(k(P'))\to K^M_{d(P'')}(k(P'')) $ where $P''=(p_0,\cdots,p_{s-1})$.
\end{itemize}

\item[(iv)] The idele class group $C(X,D)$ of $X$ relative to the effective divisor $D$ is defined as
\[ C(X,D)\coloneqq\mathrm{Coker}(\bigoplus_{P\in \mathcal{Q}}K^M_{d(P)}(k(P))\xrightarrow{Q} I(X,D)). \]
\end{itemize}
\end{definition}
\subsection{Milnor $K$-sheaves}
Let $X$ be an integral scheme. Recall the Milnor $K$-sheaf $\mathcal{K}^M_{*}$ is defined as the Nisnevich sheafification of the presheaf on affine scheme $\Spec(A)$ given as follow: \[ A \mapsto K^M_{\bullet}(A)=\bigoplus_{i\in \mathbb{N}} \underbrace{(A^{\times}\otimes_{\Z}\cdots\otimes_{\Z}A^{\times})}_{\text{$i$ times}}/I,\]
where $I$ is the two-sided ideal of the tensor algebra generated by the elements $a\otimes(1-a)$ with $a,1-a\in A^{\times}$.  This sheaf is closely related to a $p$-primary sheaf if $X$ is of characteristic $p\geq 0$, so-called logarithmic de Rham-Witt sheaf $\DRWlog{m}{r}{X}$ on the small Nisnevich (resp. \'etale) site, which is a subsheaf of $\DRW{m}{r}{X}$ (cf. \cite{illusiederham}) Nisnevich (resp. \'etale) locally generated by $d\log[x_1]_m\wedge\cdots\wedge d\log[x_r]_m$ with $x_i\in \mathcal{O}_X^{\times}$ for all $i$, $d\log[x]_m\coloneqq \frac{d[x]_m}{[x]_m}$ and $[x]_m$ is the Teichm\"uller representative of $x$ in $W_m\mathcal{O}_X$.

These notations can be generalized to a relative situation with respect to a divisor. Let $i\colon D\hookrightarrow X$ be an effective divisor with its complement $j\colon U\coloneqq X\setminus D\hookrightarrow X$.
\begin{definition}
Let  $r\in \mathbb{N}$. We define
\begin{itemize}
\item[(i)](\cite[Definition 2.4]{ruellingsaito}) the relative Milnor $K$-sheaf $\mathcal{K}^M_{r,X|D}$ on the small Nisnevich (resp. \'etale) site is defined to be the   subsheaf of $j_*\mathcal{K}^M_{r,U}$ Nisnevich  (resp. \'etale) locally generated by $\{x_1,\cdots, x_r\}$ with $x_1\in \ker(\mathcal{O}_X^{\times}\to \mathcal{O}_D^{\times})$ and $x_i\in \mathcal{O}_U^{\times}$ for all $i$. Note that if $X$ is a regular scheme over a field, then $\mathcal{K}^M_{r,X|D} \subset \mathcal{K}^M_{r,X} $ by the known Gersten conjecture \cite{kerzgersten} (see also \cite[Corollary 2.9]{ruellingsaito}).
\item[(ii)] (\cite[Definition 1.1.1]{jszduality}) in the case that $X$ is of characteristic $p\geq 0$, the relative logarithmic de Rham-Witt sheaf $\DRWlog{m}{r}{X|D} $  on the small Nisnevich (resp. \'etale) site is the subsheaf of $j_*\DRWlog{m}{r}{U}$ Nisnevich (resp. \'etale) locally generated by $d\log[x_1]_m\wedge\cdots\wedge d\log[x_r]_m$ with $x_1\in \ker(\mathcal{O}_X^{\times}\to \mathcal{O}_D^{\times})$ and $x_i\in \mathcal{O}_U^{\times}$ for all $i$. Similar to the relative Milnor $K$-group, we also have $\DRWlog{m}{r}{X|D}\subset \DRWlog{m}{r}{X} $ in the case that $X$ is a regular scheme.
\end{itemize}
 \end{definition}
 We will show relations between them in a local case, and then we may use these results in different settings. In the following, we fix the notation as follows:  Let $R$ be a henselian regular local ring of characteristic $p>0$ with the residue field $k$. We assume that $k$ is finite. Let $D$ be an effective divisor such that $C:=$ Supp($D$) is a simple normal crossing divisor on $X\coloneqq \Spec(R)$. Let $\{ D_{\lambda}\}_{\lambda\in \Lambda}$ be the (regular) irreducible components of $D$, and let $i_{\lambda}: D_{\lambda} \hookrightarrow X$ be the natural map.

 \begin{theorem}\label{relative.blochkato.general}
  		The $d\log$ map induces an isomorphism of Nisnevich sheaves on $X_{\nis}$
  		\begin{align*}
  		d\log[-]: \mathcal{K}^M_{r,X|D}/(p^m\mathcal{K}^M_{r,X}\cap \mathcal{K}^M_{r,X|D} ) &\xrightarrow{\cong} W_m\Omega^r_{X|D,\log}\\
  		\{x_1,\dots,x_r\} &\mapsto d\log [x_1]_m \wedge \cdots \wedge d\log [x_r]_m. \notag
  		\end{align*}
  	\end{theorem}
  	\begin{proof}
  		The assertion follows directly by the following commutative diagram
  		\[ \xymatrix{
  			\mathcal{K}^M_{r,X|D}/(p^m\mathcal{K}^M_{r,X}\cap \mathcal{K}^M_{r,X|D} ) \ar@{^(->}[r] \ar@{->>}[d]_{d\log}  &\mathcal{K}^M_{r,X}/p^m \ar[d]^{\cong}_{d\log}\\
  			\DRWlog{m}{r}{X|D} \ar@{^(->}[r] & 	\DRWlog{m}{r}{X},
  		}\]
  	where the right vertical map is an isomorphism by Bloch-Gabber-Kato theorem \cite{blochkato} and Gersten resolutions of $\epsilon_*\mathcal{K}^M_{r,X}$ and $\epsilon_*\DRWlog{m}{r}{X}$ from \cite{kerzgersten} and \cite{grossuwa}, here $\epsilon\colon X_{\nis} \to X_{\mathrm{Zar}}$ is the canonical map.
  	\end{proof}
  In order to study the structure of the relative logarithmic de Rham-Witt sheaves, we introduce some notions here.  We endow $\bN^{\Lambda}$ with a semi-order by
   \[ \underline n:=\nlam \geq \underline {n'}:=(n^{\prime}_{\lambda})_{\lambda\in\Lambda} \; \text{if} \; n_{\lambda} \geq n^{\prime}_{\lambda} \;\text{for all} \; \lambda\in\Lambda.\]
    For $\underline n=\nlam \in \bN^{\Lambda}$ let
   $$
   \D n=\sum\limits_{\lambda\in\Lambda}n_{\lambda}D_{\lambda}
   $$
   be the associated divisor. For $\nu \in \Lambda$ we set
   $ \delta_{\nu}=(0, \dots, 1, \dots, 0 ) \in \bN^{\Lambda}  $, where $1$ is on the $\nu$th place, and we define the following Nisnevich sheaves for $r\geq 1$
   \[  \gr^{\ul n, \nu}\mathcal{K}^M_{r,X}\coloneqq \mathcal{K}^M_{r,X|D_{\ul n}} /\mathcal{K}^M_{r,X|D_{\ul n+\delta_{\nu}}};\]
 \[  \gr^{\ul n, \nu}\DRWlog{m}{r}{X} \coloneqq
 \DRWlog{m}{r}{X|D_{\ul n}}/\DRWlog{m}{r}{X|D_{\ul n+\delta_{\nu}}}.
 \]
 \begin{proposition}{\cite[Proposition 2.10]{ruellingsaito}}\label{grad.MK}
 	Let $\underline n=\nlam \in \bN^{\Lambda}$, and let $\nu \in \Lambda, r\geq 1$. Assume $n_{\nu}=0$ and set
 		\[ D_{\nu, \ul n}:=\sum_{\lambda \in \Lambda\setminus \{\nu\}} n_{\lambda}(D_{\lambda} \cap D_{\nu}).  \]
 		Then there is a natural isomorphism of Nisnevich sheaves
 		\[ \normalfont \text{gr}^{\ul n, \nu} \mathcal{K}^M_{r,X}  \xrightarrow{\cong} i_{\nu,*}\mathcal{K}^M_{r,D_{\nu}|D_{\nu, \ul n}}.\]
 \end{proposition}
\begin{proof}
	The argument in \cite{ruellingsaito} works verbatim for our case.
\end{proof}
\begin{theorem}\label{relative.blochkato}
	If $D$ is reduced, then  $d\log$ induces an isomorphism of Nisnevich sheaves
	\begin{align*}
	d\log[-]: \mathcal{K}^M_{r,X|D}/p^m &\xrightarrow{\cong} W_m\Omega^r_{X|D,\log}\\
	\{x_1,\dots,x_r\} &\mapsto d\log [x_1]_m \wedge \cdots \wedge d\log [x_r]_m. \notag
	\end{align*}
\end{theorem}
\begin{proof}
	By the commutative diagram
	\[ \xymatrix{
		\mathcal{K}^M_{r,X|D}/p^m \ar[r] \ar@{->>}[d]_{d\log}  &\mathcal{K}^M_{r,X}/p^m \ar[d]^{\cong}_{d\log}\\
		\DRWlog{m}{r}{X|D} \ar@{^(->}[r] & 	\DRWlog{m}{r}{X}
	}\]
	it is enough to show that $	\mathcal{K}^M_{r,X|D}/p^m \hookrightarrow \mathcal{K}^M_{r,X}/p^m $.
	On the other hand, we have the following commutative diagram:
	\[ \xymatrix {
		0 \ar[r] & 	\mathcal{K}^M_{r,X|D}\ar[r]\ar[d]^{p^m} & \mathcal{K}^M_{r,X} \ar[r]\ar[d]^{p^m} & \mathcal{K}^M_{r,X}/\mathcal{K}^M_{r,X|D}\ar[r]\ar[d]^{p^m} & 0 \\
		0 \ar[r] & 	\mathcal{K}^M_{r,X|D}\ar[r] & \mathcal{K}^M_{r,X} \ar[r] & \mathcal{K}^M_{r,X}/\mathcal{K}^M_{r,X|D}\ar[r] & 0.
	}\]
	Combining the fact \cite[Theorem 8.1]{geisserlevine} and the Gersten resolution \cite{kerzgersten}, we know that $ \mathcal{K}^M_{r,X}$ is $p$-torsion free. Therefore the middle vertical map is injective, so is the first vertical map.
	By the snake lemme, it is sufficient to check that the third vertical map	$p^m: \mathcal{K}^M_{r,X}/\mathcal{K}^M_{r,X|D} \to \mathcal{K}^M_{r,X}/\mathcal{K}^M_{r,X|D}$ is injective. This follows from the above Proposition \ref{grad.MK}, by noting that $\mathcal{K}^M_{r,X}/\mathcal{K}^M_{r,X|D} $ is a successive extension of sheaves $\gr^{\ul n, \nu} \mathcal{K}^M_{r,X} $ and the map $p^m: i_{\nu,*}\mathcal{K}^M_{r,D_{\nu}|D_{\nu, \ul n}} \to i_{\nu,*}\mathcal{K}^M_{r,D_{\nu}|D_{\nu, \ul n}}$ is injective (similar to the injectivity of the first vertical map in above diagram). We remark that the assumption in Proposition \ref{grad.MK} is satisfied, since $D$ is reduced.
\end{proof}

 \begin{proposition}{\cite[Proposition 1.1.9]{jszduality}}\label{gr.log.form}
Let $X,D$ be as above. Then we have
 \begin{itemize}
 \item[(i)] $W_m\Omega_{X,\log}^d=W_m\Omega_{X|D_{\mathrm{red}},\log}^d$;
 \item[(ii)] for $\ul{n}\geq \ul{1}$, the quotient
 $ \normalfont \gr^{\underline{n}, \nu}\DRWlog{m}{r}{X}$ is a coherent $\cO_{D_{\nu}}^{p^e}$-module, for some $e>>0$.
 \end{itemize}
 \end{proposition}
\begin{proof}
	In the case that $d=1$ (i.e., $R$ is a discrete valuation ring), the assertions have been given in \cite[(4.7),(4.8)]{blochkato}.
	For general $d$, in \cite{jszduality}, the graded pieces have been studied in the case that $R$ is the henselization of a local ring of a smooth scheme over $k$. But note that the argument also works in our setting. We only need to show (i).  By Theorem \ref{relative.blochkato}, we see that, for $\ul{n}<\ul{1}$, \[ \gr^{\ul n, \nu} \mathcal{K}^M_{d,X} /p^m \cong i_{\nu,*}\mathcal{K}^M_{d,D_{\nu}|D_{\nu, \ul n}}/p^m = i_{\nu,*}\DRWlog{m}{d}{D_{\nu}|D_{\nu, \ul n}}=0,\]
	where the vanishing is by dimension.
\end{proof}

\section{Class field theory for proper varieties over finite fields}\label{sec:cftvarieties}
\bigskip
In this section we reprove the main results of the class field theory of smooth proper
varieties over finite fields with ramification along divisors $D$, which originally are due to Kato-Saito \cite{katosaito}.

Let $X$ be a smooth proper variety of dimension $d$ over a finite field $k$,  let $D$ be an effective divisor such that $C:=$ Supp($D$) is a simple normal crossing divisor on $X$, and let $j: U:=X-C \hookrightarrow X$ be the complement of $C$. Let $\{ D_{\lambda}\}_{\lambda\in \Lambda}$ be the (smooth) irreducible components of $D$, and let $i_{\lambda}: D_{\lambda} \hookrightarrow X$ be the natural map. We use the dimension function $d(x)=\mathrm{dim}(\overline{\{x\}})$ for $x\in X$. We also denote $X_r\coloneqq \{ x\in X|\  d(x)=r\}$ the set of points of dimension $r$ of $X$ and $X^r\coloneqq X_{d-r}$ the set of points of codimension $r$ of $X$.

\subsection{Idele class groups}  The $K$-theoretic class group $H^d(X_{\nis}, \mathcal{K}_{d,X|D}^M)$ is introduced by Kato-Saito in \cite{katosaito}, and they also give an idelic description of the dual of this class group.  In \cite{kerzideles}, we give a direct description of this class group, and prove the following theorem.
\begin{theorem}{(\cite[Theorem 8.4]{kerzideles})}\label{higheridele}
There exists a unique isomorphism
\[ \rho_{X,D}\colon C(X,D) \cong H^d(X_{\nis}, \mathcal{K}^M_{d, X|D})\]
such that the following triangle commutes
\[ \xymatrix{
&\bigoplus_{x\in X_0}\Z \ar[dl]_{\imath} \ar[dr]^-{\imath_{\nis}}&\\
C(X,D)\ar[rr]^-{\rho_{X,D}}&& H^d(X_{\nis}, \mathcal{K}^M_{d, X|D}),
}\]
where $\imath$ is the obvious map, and $\imath_{\nis}$ is the map from \cite[Theorem 2.5]{katosaito}.
\end{theorem}

\subsection{The $\ell$-primary part}\label{l-part.over.finite fields}

In this subsection, we study the group $H^d(X_{\nis},\mathcal{K}^M_{d, X|D})/\ell^m$, and compare it with $H^{2d}(X_{\et}, j_!\mu_{\ell^m}^{\otimes d})$.

The coniveau spectral sequence for an abelian \'etale (resp. Nisnevich) sheaf $\cF $ on $X_{\et}$ (resp. $X_{\nis}$)  writes
\[  E^{p,q}_{1,\et}(\cF):=\bigoplus_{x\in X^p} H^{p+q}_x(X_{\et}, \cF)\Longrightarrow H^{p+q}(X_{\et}, \cF) \]
\[  E^{p,q}_{1,\nis}(\cF):=\bigoplus_{x\in X^p} H^{p+q}_x(X_{\nis}, \cF)\Longrightarrow H^{p+q}(X_{\nis}, \cF) ,\]
where $X^p$ is the set of points of codimension $p$ of $X$.
Note that the degeneration of the coniveau spectral sequence due to cohomological dimension (cf. \cite[1.2.5]{katosaito}) for $\cRM{d} $ on $X_{\nis}$ gives rise to a short exact sequence
\begin{equation}\label{nis.ladic.exact}
 \bigoplus_{x\in X^{d-1}}H^{d-1}_x(X_{\nis}, \cRM{d})\to \bigoplus_{x\in X^{d}}H^{d}_x(X_{\nis}, \cRM{d}) \to H^{d}(X_{\nis}, \cRM{d}) \to 0.
\end{equation}
We now study the coniveau spectral sequence for $j_!\mu_{\ell^m}^{\otimes d}$ on $X_{\et}$
\begin{proposition}\label{ladic.top}
Let $X$ be a smooth (not necessarily proper) variety over a finite field of dimension $d$.  For any $x\in X^a$, we have \[ H^{a+d+1}_x(X_{\et}, j_!\mu_{\ell^m}^{\otimes d})=H^{a+d+1}_x(X_{\et},\mu_{\ell^m}^{\otimes d}),\]i.e., $E_{1,\et}^{\bullet,d+1}(j_!\mu_{\ell^m}^{\otimes d})=E_{1,\et}^{\bullet,d+1}(\mu_{\ell^m}^{\otimes d})$. In particular, we have $E_{2,\et}^{d-2,d+1}(j_!\mu_{\ell^m}^{\otimes d})=E_{2,\et}^{d-1,d+1}(j_!\mu_{\ell^m}^{\otimes d})=0$.
\end{proposition}
\begin{proof}
We prove the first claim by induction on the codimension $a$. For $x\in X^a$, we denote $X_x=\Spec(\cO_{X,x}^{h})$ the henselization of $X$ at $x$, and $Y_x=X_x\setminus\{x\}$.
If $a=1$, then any divisor of $X_x$ must have support in the closed point $\{x\}$. Therefore  \[ j_!\mu_{\ell^m}^{\otimes d}|_{Y_x} = \mu_{\ell^m}^{\otimes d}|_{Y_x} \] by the definition of $j_!$. Using the localization exact sequences twice, we obtain
\[ H^{d+2}_x(X_{\et}, j_!\mu_{\ell^m}^{\otimes d})\cong H^{d+1}(Y_{x,\et},j_!\mu_{\ell^m}^{\otimes d})=H^{d+1}(Y_{x,\et},\mu_{\ell^m}^{\otimes d}) \cong H^{d+2}_x(X_{\et},\mu_{\ell^m}^{\otimes d}),\]
where the first isomorphism is due to $j_!\mu_{\ell^{m}}^{\otimes d}|_x=0$, and the second isomorphism is by the vanishing $H^{d+2}(X_{x,\et},\mu_{\ell^{m}}^{\otimes d})\cong H^{d+2}(x_{\et},\mu_{\ell^{m}}^{\otimes d})=0=H^{d+1}(x_{\et},\mu_{\ell^{m}}^{\otimes d})\cong H^{d+1}(X_{x,\et},\mu_{\ell^{m}}^{\otimes d})$, where we use the fact that $\mathrm{cd}_{\ell}(x)\leq d+1-\mathrm{codim}_X(x)$ (cf. \cite[Lemma 4.2(1)]{satoladiccft}).

For general codimension $a>1$, the coniveau spectral sequence on $Y_x$ and cohomological vanishing give us an exact sequence
\begin{align}
 \bigoplus_{y\in Y_x^{a-2}}H^{a+d-1}_y(Y_{x,\et},j_!\mu_{\ell^m}^{\otimes d}) \to \bigoplus_{y\in Y_x^{a-1}}H^{a+d}_y(Y_{x,\et},j_!\mu_{\ell^m}^{\otimes d}) \to H^{a+d}(Y_{x,\et}, j_!\mu_{\ell^m}^{\otimes d}) \to 0.
\end{align}
On the other hand, the localization exact sequence for $j_!\mu_{\ell^m}^{\otimes d}$ on $X_x$ tells us
\begin{align}\label{ladic.long.exact}
H^{a+d}(Y_{x,\et}, j_!\mu_{\ell^m}^{\otimes d}) \cong H^{a+d+1}_x(X_{x,\et}, j_!\mu_{\ell^m}^{\otimes d}),
\end{align}
Indeed due  to $\mathrm{cd}_{\ell}(x)\leq d+1-\mathrm{codim}_X(x)$ we have $$H^{a+d}(X_{x,\et}, j_!\mu_{\ell^m}^{\otimes d} )=0=H^{a+d+1}(X_{x,\et}, j_!\mu_{\ell^m}^{\otimes d} ).$$

Combining these facts, we get the following diagram with exact rows
\[ \xymatrix@C=1em{
\bigoplus\limits_{y\in Y_x^{a-2}}H^{a+d-1}_y(Y_{x,\et},j_!\mu_{\ell^m}^{\otimes d}) \ar[r]\ar[d] &  \bigoplus\limits_{y\in Y_x^{a-1}}H^{a+d}_y(Y_{x,\et},j_!\mu_{\ell^m}^{\otimes d}) \ar[r]\ar[d] &H^{a+d+1}_x(X_{x,\et},j_!\mu_{\ell^m}^{\otimes d}) \ar[r]\ar[d] &0\\
\bigoplus\limits_{y\in X_x^{a-2}}H^{a+d-1}_y(X_{x,\et},\mu_{\ell^m}^{\otimes d}) \ar[r] & \bigoplus\limits_{y\in X_x^{a-1}}H^{a+d}_y(X_{x,\et},\mu_{\ell^m}^{\otimes d}) \ar[r] &H^{a+d+1}_x(X_{x,\et},\mu_{\ell^m}^{\otimes d}) \ar[r] &0
}\]
The first two vertical maps are isomorphisms by induction. Hence the third vertical arrow is also an isomorphism. Thanks to \cite[Theorem 3.5.1]{jssduality}, we see that the complex $E_{1,\et}^{\bullet,d+1}(\mu_{\ell^m}^{\otimes d})$ is the Kato complex of $\mu_{\ell^m}^{\otimes d}$  (cf. \cite[(0.2)]{kerzsaitoIHES}) up to a sign.  By the known Kato conjecture on vanishing of cohomology groups of this complex  at places $d-1$ and $d-2$
(cf. \cite[Theorem 8.1]{kerzsaitoIHES}) we obtain the second part of Proposition~\ref{ladic.top}.
\end{proof}

\begin{corollary}\label{et.ladic.exact}
We have the following exact sequence
\begin{equation*}
\bigoplus_{x\in X^{d-1}}H^{2d-1}_x(X_{\et}, j_!\mu_{\ell^m}^{\otimes d})\to \bigoplus_{x\in X^{d}}H^{2d}_x(X_{\et}, j_!\mu_{\ell^m}^{\otimes d}) \to H^{2d}(X_{\et}, j_!\mu_{\ell^m}^{\otimes d})\to 0.
\end{equation*}
\end{corollary}
\begin{proof}
By the above proposition, we have $E_{2,\et}^{d,d}(j_!\mu_{\ell^m}^{\otimes d})=H^{2d}(X_{\et}, j_!\mu_{\ell^m}^{\otimes d})$.
\end{proof}

Using the Galois symbol maps and induction on codimension, Sato constructs the localized Chern class map and proves the following theorem.

\begin{theorem}[{\cite[Theorem 1.2 and \S 3]{satoladiccft}}]\label{Sato.ladic}
For any $x\in X^a$, there exists a canonical surjective map \[ \mathrm{cl}_{X,D,x,\ell^m}^{d,\mathrm{loc}} \colon H^a_x(X_{\nis}, \cRM{d})/\ell^m \twoheadrightarrow H^{d+a}_x(X_{\et}, j_!\mu_{\ell^m}^{\otimes d}),\]
which is called localized Chern class map. Moreover, if $x\in X^d$, the localized Chern class map
\[ \mathrm{cl}_{X,D,x,\ell^m}^{d,\mathrm{loc}} \colon H^d_x(X_{\nis}, \cRM{d})/\ell^m \xrightarrow{\cong} H^{2d}_x(X_{\et}, j_!\mu_{\ell^m}^{\otimes d})\]
is bijective.
\end{theorem}

\begin{corollary}\label{ladic.nis=et}
There is a canonical isomorphism
 \begin{equation*}
 H^d(X_{\nis},\cRM{d})/\ell^m \cong H^{2d}(X_{\et}, j_!\mu_{\ell^m}^{\otimes d}).
 \end{equation*}
\end{corollary}
\begin{proof}
We have the following commutative diagram with exact rows:
\[ \xymatrix@C=3mm{ \bigoplus\limits_{x\in X^{d-1}}H^{d-1}_x(X_{\nis}, \cRM{d})/\ell^m \ar[r]\ar@{->>}[d]^-{\mathrm{cl}_{X,D,x,\ell^m}^{d,\mathrm{loc}}}&\bigoplus\limits_{x\in X^{d}}H^{d}_x(X_{\nis}, \cRM{d})/\ell^m \ar[r]\ar[d]^-{\mathrm{cl}_{X,D,x,\ell^m}^{d,\mathrm{loc}}}_-{\cong}& H^{d}(X_{\nis}, \cRM{d})/\ell^m \ar[r]\ar@{-->}[d]& 0\\
\bigoplus\limits_{x\in X^{d-1}}H^{2d-1}_x(X_{\et}, j_!\mu_{\ell^m}^{\otimes d})\ar[r] &\bigoplus\limits_{x\in X^{d}}H^{2d}_x(X_{\et}, j_!\mu_{\ell^m}^{\otimes d}) \ar[r] &H^{2d}(X_{\et}, j_!\mu_{\ell^m}^{\otimes d})\ar[r] &0,
}\]
where the first exact row follows from the exact sequence (\ref{nis.ladic.exact}) by tensoring with $\Z/\ell^m\Z$, the second is Corollary \ref{et.ladic.exact}.  By Theorem \ref{Sato.ladic} the first vertical arrow is surjective and the second is bijective. Then the assertion follows from an easy diagram chasing.
\end{proof}
\begin{theorem}[{\cite[Lemma 2.9]{saito1989}}]
There is a perfect pairing of finite $\Z/\ell^m\Z$-modules
\begin{equation*}
\normalfont H^i(U_{\et}, \mu_{\ell^m}^{\otimes r}) \times H^{2d+1-i}(X_{\et}, j_!\mu_{\ell^m}^{\otimes d-r}) \to H^{2d+1}(X_{\et}, j_!\mu_{\ell^m}^{\otimes d}) \xrightarrow{\cong} \Z/\ell^m\Z.
\end{equation*}
\end{theorem}
In particular, in case $i=1, r=0$,  we obtain
\begin{equation}
H^d(X_{\et},j_!\mu_{\ell^m}^{\otimes d})/\ell^m \cong \pi^{\mathrm{ab}}_1(U)/\ell^m.
\end{equation}
In summary:
\begin{corollary}\label{ladic_cft}
We obtain canonical isomorphisms
\begin{equation*}
\normalfont C(X,D)/\ell^m\stackrel{\rho_{X,D}}{\cong} H^d(X_{\nis},\mathcal{K}^M_{d, X|D})/\ell^m \cong \pi^{\mathrm{ab}}_1(U)/\ell^m.
\end{equation*}
\end{corollary}

\subsection{The $p$-primary part}
\bigskip
In this subsection we want to compare the group $H^d(X_{\nis}, \cRM{d})/p^m$ with the group $H^d(X_{\et}, \cRF)$.

The coniveau spectral sequence for a $p$-primary \'etale (resp. Nisnevich) sheaf $\cF $ on $X_{\et}$ (resp. $X_{\nis}$)  writes
\[  E^{p,q}_{1,\et}(\cF):=\bigoplus_{x\in X^p} H^{p+q}_x(X_{\et}, \cF)\Longrightarrow H^{p+q}(X_{\et}, \cF) \]
\[  E^{p,q}_{1,\nis}(\cF):=\bigoplus_{x\in X^p} H^{p+q}_x(X_{\nis}, \cF)\Longrightarrow H^{p+q}(X_{\nis}, \cF) .\]

We know that $E^{p,q}_{1,\et}(\cF)=0$ if $q>1 $ or $p>d$, and $E^{p,q}_{1,\nis}(\cF)=0$ if $q>0$ or $p>d$.

\begin{theorem}\label{nis=et}
The canonical map \[ \normalfont H^d(X_{\nis}, \cRF)\xrightarrow{\cong} H^d(X_{\et}, \cRF)\]
is an isomorphism.
\end{theorem}
\begin{proof}
By the coniveau spectral sequences, it follows from the following two propositions.
\end{proof}
\begin{proposition}\label{E1acyclic}
	Let $X$ be a smooth (not necessarily proper) variety over a finite field of dimension $d$.
The map $ \normalfont E^{\bullet, 1}_{1,\et}(\cRF)\xrightarrow{\cong} E^{\bullet, 1}_{1,\et}(\cAF)$ is an isomorphism of complexes. Therefore we have $\normalfont E^{d-1, 1}_{2,\et}(\cRF)=E^{d-2, 1}_{2,\et}(\cRF)=0$.
\end{proposition}
\begin{proof}
For $x\in X^a$, we denote $X_x:=\Spec(\cO_{X,x}^{h})$ the henselization of $X$ at $x$, and $Y_x:=X_x\setminus \{x\}$.  We want to prove that
\[ H^{a+1}_x(X,\cRF) \cong H^{a+1}_x(X, \cAF).\]

If $a=1$, then any divisor of $X_x$ must have support in the closed point $\{x\}$. Therefore, we have \[  \cRF|_{Y_x} = \cAF|_{Y_x} \]
by the definition of $\cRF$. Using the localization exact sequences twice, we obtain
\[ \xymatrix{ H^1(X_{x,\et},\cRF) \ar[r]\ar[d] & H^1(Y_{x,\et},\cRF)\ar[r]\ar@{=}[d] & H^2_x(X_{\et},\cRF)\ar[r]\ar[d]&0\\
H^1(X_{x,\et},\cAF) \ar[r] & H^1(Y_{x,\et},\cAF)\ar[r] & H^2_x(X_{\et},\cAF)\ar[r]&0.
}\]
We claim that the first vertical arrow is surjective: Indeed, we have the  exact sequence
\[
  H^1(X_{x,\et},\cRF) \to H^1(X_{x,\et},\cAF)\to H^1(X_{\et}, \cAF/\cRF), \]
where  $H^1(X_{\et}, \cAF/\cRF)=0$ since this sheaf is a successive extension of coherent
sheaves by Proposition \ref{gr.log.form}. We conclude that the third vertical map in the
previous commutative diagram is an isomorphism.

For general codimension $a>1$, we proof this by induction. The coniveau spectral sequence on $Y_x$ gives us the exact sequence
\begin{equation}
 \footnotesize \bigoplus\limits_{y\in Y_x^{a-2}}H^{a-1}_y(Y_{x,\et},\cRF) \to \bigoplus\limits_{y\in Y_x^{a-1}}H^{a}_y(Y_{x,\et},\cRF) \to H^a(Y_{x,\et}, \cRF) \to 0.
\end{equation}
On the other hand, the  localization exact sequence for $\cRF$ on $X_x$ tells us
\begin{align}
H^a(Y_{x,\et}, \cRF) \cong H^{a+1}_x(X_{x,\et},\cRF),
\end{align}
since we know that $H^{a+1}(X_{x,\et},\cRF)\cong H^{a+1}(x_{\et},\cRF)=0$ and similarly $H^{a}(X_{x,\et},\cRF)\cong H^{a}(x_{\et},\cRF)=0$.
Combining these facts, we get the following diagram with exact rows:
\[ \xymatrix@C=1em{
\scriptstyle\bigoplus\limits_{y\in Y_x^{a-2}}H^{a-1}_y(Y_{x,\et},\cRF) \ar[r]\ar[d] & \scriptstyle \bigoplus\limits_{y\in Y_x^{a-1}}H^{a}_y(Y_{x,\et},\cRF) \ar[r]\ar[d] &H^{a+1}_x(X_{x,\et},\cRF) \ar[r]\ar[d] &0\\
\scriptstyle\bigoplus\limits_{y\in X_x^{a-2}}H^{a-1}_y(X_{x,\et},\cAF) \ar[r] & \scriptstyle \bigoplus\limits_{y\in X_x^{a-1}}H^{a}_y(X_{x,\et},\cAF) \ar[r] &H^{a+1}_x(X_{x,\et},\cAF) \ar[r] &0
}\]
The first two vertical maps are isomorphisms by induction. Hence the third vertical arrow is also an isomorphism. Thanks to \cite[Theorem 4.11.1]{jssduality}, we see that the complex $E_{1,\et}^{\bullet,1}(\DRWlog{m}{d}{X})$ is the Kato complex of $\cAF$ (cf. \cite[(0.2)]{kerzsaitoIHES}) up to a sign.  By the known Kato conjecture on vanishing of the cohomology groups of this complex at places $d-1$ and $d-2$ (cf. \cite{jskatohomology}), we obtain the second part of Proposition \ref{E1acyclic}.
\end{proof}
\begin{proposition}\label{E0nis=et}
Let $X$ be a smooth (not necessarily proper) over a finite field $k$ of dimension $d$. For any $x\in X^a$, the canonical map
\begin{equation}\label{loc.nis=et}
\normalfont H^{a}_x(X_{\nis},\cRF) \to H^{a}_x(X_{\et}, \cRF)
\end{equation}
is an isomorphism.

That is, there is a natural isomorphism of complexes
\[ \normalfont E^{\bullet,0}_{1,\nis}(\cRF)\xrightarrow{\cong} E^{\bullet,0}_{1,\et}(\cRF) \]
\end{proposition}
\begin{proof}

To prove this, we use Proposition \ref{gr.log.form}(ii). We reduced to the case that $D$ is reduced, since the quotient $W_m\Omega_{X|D}^d/W_m\Omega_{X|D_{\mathrm{red}}}^d$ on $X_{\nis}$ is a successive extension of coherent sheaves, for which the \'etale and Nisnevich cohomology groups are the same.  By Proposition \ref{gr.log.form}(i), it is equivalent to show that the canonical map
\[ H^{a}_x(X_{\nis},\cAF) \xrightarrow{\cong} H^{a}_x(X_{\et}, \cAF)\]
is an isomorphism. This is true since both are isomorphic to $K^M_{d-a}(k(x))/p^m=\DRWlog{m}{d-a}{x}$ by purity \cite[Proposition 2.1]{milneduality} and the known Gersten conjecture \cite{grossuwa}.

\end{proof}
\begin{corollary}\label{mod.pm}
There is a canonical isomorphism
\begin{equation*}
 \normalfont H^d(X_{\nis}, \cRM{d})/p^m \cong H^d(X_{\et}, W_m\Omega_{X|D,\log}^d).
 \end{equation*}
\end{corollary}
\begin{proof}
First we have
\[ H^d(X_{\nis}, \cRM{d})/p^m \cong H^d(X_{\nis}, \cRM{d}/p^m) \cong H^d(X_{\nis}, \cRM{d}/p^m\mathcal{K}^M_{d, X}\cap \cRM{d}), \]
where the first isomorphism is due to the fact that the Nisnevich cohomological dimension
of $X$ is $d$, and the second follows from the observation that the support of
$p^m\mathcal{K}^M_{d, X}\cap \cRM{d}/p^m\mathcal{K}^M_{d, X|D}$ is contained in $D$, which is of dimension $d-1$.

By Theorem \ref{relative.blochkato.general} and Theorem \ref{nis=et}, hence we have
\[ H^d(X_{\nis}, \cRM{d})/p^m \cong H^d(X_{\nis}, W_m\Omega_{X|D,\log}^d) \cong H^d(X_{\et}, W_m\Omega_{X|D,\log}^d). \]
\end{proof}
\begin{corollary}\label{transition.surjective}
Let $D_1,D_2$ be two effective divisors on $X$ whose supports are simple normal crossing divisors. Assume $D_1\geq D_2$. Then the canonical map
\[ \normalfont  H^d(X_{\et}, W_m\Omega_{X|D_1,\log}^d) \to H^d(X_{\et}, W_m\Omega_{X|D_2,\log}^d) \]
is surjective.
\end{corollary}
\begin{proof}
Note that we have the following exact sequence on $X_{\nis}$
\[ 0\to \mathcal{K}^M_{d, X|D_1} \to  \mathcal{K}^M_{d, X|D_2} \to \mathcal{K}^M_{d, X|D_2}/\mathcal{K}^M_{d, X|D_1}\to 0, \]
but the Nisnevich sheaf $\mathcal{K}^M_{d, X|D_2}/\mathcal{K}^M_{d, X|D_1}$ is supported in $D_2$, which is of dimension $d-1$. Hence the associated long exact sequence implies that
\[ H^d(X_{\nis},\mathcal{K}^M_{d, X|D_1} ) \to H^d(X_{\nis},\mathcal{K}^M_{d, X|D_2} )   \]
is surjective. Therefore the claim follows from Corollary \ref{mod.pm}.
\end{proof}

Now, we recall the duality theorem of the relative logarithmic de Rham-Witt sheaves.
\begin{theorem}(\cite[Theorem 4.1.4]{jszduality})\label{jszduality}
Let $X,U, D$ be as before. For $i\in \mathbb{N}, r\in \mathbb{N}$, there are natural perfect pairings of topological groups
\[\normalfont H^i(U_{\et}, \DRWlog{m}{r}{U}) \times \varprojlim\limits_{\stackrel{E}{\mathrm{Supp}(E)\subset X\setminus U}} H^{d+1-i}(X_{\et},\DRWlog m {d-r} {X|E} ) \to H^{d+1}(X_{\et}, \cAF) \xrightarrow{\mathrm{Tr}} \Zpm, \]
where the first group is endowed with discrete topology, the second is endowed with profinite topology, and the limit with respect to all effective divisor $E$ with $\Supp(E)\subset X\setminus U$.
\end{theorem}
In particular, for $i=1$ and $r= 0$ we get isomorphisms
	\begin{equation*}
	\varprojlim\limits_{E} H^d(X_{\et},\DRWlog m d {X|E})\xrightarrow{\cong} H^1(U_{\et},\Z/p^m\Z)^\vee \cong \pi^{ab}_1(U)/p^m,
	\end{equation*}
	and
	\begin{equation*}
	H^1(U_{\et},\Z/p^m\Z) \xrightarrow{\cong} \varinjlim\limits_{E} H^d(X_{\et},\DRWlog m d {X|E})^{\vee},
	\end{equation*}
	where $A^\vee$ is the Pontryagin dual of a topological abelian group $A$.
	These isomorphisms can be used to define a measure of ramification for  \'{e}tale abelian covers of $U$ whose degree divides $p^m$.
	\begin{definition}\label{new.filtration}
		For our divisor $D$, we define
		\[  \normalfont \Fil_DH^1(U_{\et},\Z/p^m\Z):= H^d(X_{\et},\DRWlog m d {X|D})^{\vee}.\]

		Dually we define
		\[  \normalfont \pi^{\text{ab}}_1(X,D)/p^m :=\Hom(\Fil_DH^1(U_{\et},\Z/p^m\Z),\Z/p^m\Z). \]
		\end{definition}
	The group $ \pi^{\text{ab}}_1(X,D)/p^m$ is a quotient of $\pi_1^{\text{ab}}(U)/p^m$, which can be thought of as classifying abelian \'etale coverings of $U$ whose degree divides $p^m$ with ramification bounded by $D$.
\begin{corollary}\label{cft.mod.pm}
We have canonical isomorphisms
\begin{equation*}
\normalfont C(X,D)/p^m\cong H^d(X_{\nis}, \cRM{d})/p^m \xrightarrow{\cong} \pi^{\text{ab}}_1(X,D)/p^m.
 \end{equation*}
\end{corollary}
\begin{proof}
This is a consequence of Theorem \ref{jszduality} and Corollary \ref{mod.pm}.
\end{proof}

\subsection{Class field theory via ideles}

\begin{theorem}{(logarithmic version of wildly ramified class field theory)}\label{thm.logcftvar}
For any integer $n$, there exists a canonical isomorphism
\[  \rho_{X,D,n}\colon C(X,D)/n \xrightarrow{\cong} \pi_1^{\mathrm{ab}}(X,D)/n,   \]
such the following triangle commutes
\[ \xymatrix{
&\bigoplus_{x\in X_0}\Z \ar[dl]_{\imath} \ar[dr]^-{\rho_U}&\\
C(X,D)/n \ar[rr]^-{\rho_{X,D}/n}&& \pi^{\mathrm{ab}}_1(U)/n
}\]
where the right diagonal map $\rho_U$ sends $1$ at the point $x$ to the Frobenius ${\rm Frob}_x$.
In particular, $ \rho_{X,D,n}$ induces an isomorphism
\begin{equation}
\varprojlim_{D,n}C(X,D)/n \cong \pi^{\mathrm{ab}}_1(U).
\end{equation}
\end{theorem}

\begin{proof}
For $n=p^m$, this follows  from Corollary \ref{cft.mod.pm} and Theorem \ref{higheridele} directly. For $n$ prime to $p$, this is Corollary \ref{ladic_cft}.
\end{proof}

\begin{remark}
  The wildly ramified class field theory in \cite{kerzsaitochowgroup}, where we work with
  the relative Chow group of zero cycles instead of the idelic class group, comprises Theorem~\ref{thm.logcftvar}.
\end{remark}

\bigskip
\section{Class field theory for complete local rings over $\mathbb{F}_q$}\label{sec:cftlocalrings}
\bigskip
Let $(A,\m)$ be a complete regular local ring of dimension $d$ and of characteristic $p>0$, and let $k:=A/\m$ be the residue field. We assume that $k$ is finite. We denote $X=\Spec(A), x=\m \in X$. Let $D$ be an effective divisor with $\Supp(D)$ is a simple normal crossing divisor, let $U=X\setminus D$ be its complement. Set $X'=X\setminus \{x\}, D'=D\setminus\{x\}$. We use the dimension function on $X$ (hence also induces one on $X'$) by $d(x)=\mathrm{dim}(\overline{\{x\}})$.
\subsection{Grothendieck's local duality}
We know that the sheaf $\Omega_X^d$ is a dualizing sheaf of $X$. There exists a natural homomorphism called the residue homomorphism \cite[ \S 5]{kunzresidues}:
\[  \mathrm{res}\colon H^d_x(X, \Omega_X^d)  \to k.\]
By compositing with the trace map $\mathrm{Tr}_{k/\Fp}\colon k \to \Fp=\Z/p\Z$, we get the map
\[  \mathrm{Tr}_{k/\Fp}\circ \mathrm{res}\colon H^d_x(X, \Omega_X^d)  \to \Z/p\Z. \]

For any finite $A$-module $M$, the Yoneda pairing and the above trace map give us a canonical pairing
\begin{equation}\label{yoneda.pairing}
 H^i_x(X, M)  \times \mathrm{Ext}_X^{d-i}(M, \Omega_X^d) \to \Z/p\Z.
 \end{equation}
\begin{theorem}[Grothendieck local duality \cite{hartshornelocal}]\label{grothendieck.local}
For each integer $i\geq 0$, the pairing (\ref{yoneda.pairing}) induces the isomorphisms
$$ \text{Ext}^{d-i}_A(M, \Omega_X^d)  \cong \text{Hom}_{\Zp1}(H_{x}^i(X, M), \Zp1),$$
$$ H_{x}^i(X, M) \cong  \text{Hom}_{\text{cont}}(\text{Ext}^{d-i}_A(M, \Omega_X^d), \Zp1),$$
where $\text{Hom}_{\text{cont}}$ denotes the set of continuous homomorphisms with respect to $\m$-adic topology on Ext group.
\end{theorem}

In particular, if $M$ is a locally free $A$-module, we obtain the isomorphisms
\begin{equation}
H^{d-i}(M^{t}) \cong \text{Hom}_{\Zp1}(H_{x}^i(X, M), \Zp1),
\end{equation}
where $M^t\coloneqq Hom_A(M, \Omega_X^d)$ is the dual $A$-module, and
\begin{equation}
 H_{x}^i(X, M) \cong  \text{Hom}_{\text{cont}}(H^{d-i}(M^{t}), \Zp1).
\end{equation}
Note that, for a locally free $A$-module $M$, we have \cite{hartshornelocal}
\begin{equation}\label{loc.are.cm}
H_{x}^i(X, M)=0 \   \  \mathrm{if} \ \ \   i\neq d.
\end{equation}

\subsection{Duality theorems}

The purity result of Shiho \cite[Theorem 3.2]{shihopurity} tells us that there exists a canonical isomorphism
\begin{equation}\label{trace.map}
 \mathrm{Tr}\colon H^{d+1}_x(X_{\et}, \DRWlog m d X) \xrightarrow{\cong} H^1(x, \Z/p^m\Z) \cong \Z/p^m\Z.
\end{equation}
Using the same method as in \cite{zhaoduality}, we obtain a map
\begin{equation*}
\Phi_{m}^{i,r}\colon H^i(U_{\et},\DRWlog{m}{r}{U}) \to \varinjlim\limits_E \mathrm{Hom}_{\Z/p^n\Z}(H^{d+1-i}_x(X_{\et}, \DRWlog m {d-r} {X|E}), H^{d+1}_x(X_{\et}, \DRWlog m d X) ).
\end{equation*}
If we endow $H^i(U_{\et},\DRWlog{m}{r}{U})$ with the discrete topology and endow $\varprojlim\limits_EH^{d+1-i}_x(X_{\et}, \DRWlog m {d-r} {X|E})$ with the profinite topology, where $E$ runs over the set of effective divisors with support on $X\setminus U$, then the (continuous) map $\Phi_{m}^{i,r}$ and the trace map (\ref{trace.map})  induce a pairing of topological abelian groups:
\begin{equation}\label{pairing.mod.p}
H^i(U_{\et},\DRWlog{m}{r}{U}) \times \varprojlim\limits_{E} H^{d+1-i}_x(X_{\et}, \DRWlog m {d-r} {X|E}) \to  \Z/p^m\Z.
\end{equation}
Using Pontryagin duality, we see that  $\Phi_{m}^{i,r}$ is an isomorphism if and only if the pairing (\ref{pairing.mod.p}) is  a perfect pairing of topological abelian groups for the respective $i,m,r$.

\begin{theorem}\label{ramified.duality} For any integers $r\geq 0, m\geq 1$,  the maps $\Phi_{m}^{i,r}$ are isomorphisms.

\end{theorem}
\begin{proof}
We are reduced to the case $m=1$ by induction on $m$ and the following two exact sequences on the small \'etale site
\[ 0\to \DRWlog{m-1}{r}{U} \xrightarrow{\cdot p} \DRWlog{m}{r}{U} \xrightarrow{R} \Omega_{U,\log}^{r} \to 0  \] 
and
\[ 0\to \DRWlog{m-1}{d-r}{X|[E/p]} \xrightarrow{\cdot p} \DRWlog{m}{d-r}{X|E} \xrightarrow{R} \Omega_{X|E,\log}^{d-r}\to 0,  \] 
where $[E/p]=\sum_{\lambda\in \Lambda}[n_{\lambda}/p]D_{\lambda}$ if $D=\sum_{\lambda\in \Lambda}n_{\lambda}D_{\lambda}$, here $[n/p]={\rm min}\{ n'\in \Z| pn'\geq n\}$, and the exactness of the second complex follows from \cite[Theorem 1.1.6]{jszduality}.

Using the relation between logarithmic forms and differential forms (\cite[0, Corollary 2.1.18]{illusiederham} and \cite[Theorem 1.2.1]{jszduality}),  we see that the assertion for $i\neq 0, 1$ follows from the vanishing (\ref{loc.are.cm}) directly.  We have the following diagram with exact rows
\[ \xymatrix@C=0.25cm{ 0 \ar[r] &H^0(U_{et}, \Omega_{U,\log}^{r}) \ar[r]\ar[d] & H^0(U, Z\Omega_{U}^{r}) \ar[r] \ar[d] & H^0(U, \Omega_{U}^{r}) \ar[d]\ar[r] & H^{1}(U_{et}, \Omega_{X,\log}^{r}) \ar[r]\ar[d] &0\\
0\ar[r] & \scriptstyle \varinjlim\limits_EH_x^{d+1}(X_{\et}, \Omega_{X|E,\log}^{d-r})^{\ast} \ar[r] & \scriptstyle \varinjlim\limits_EH_x^d(X_{\et}, \Omega_{X|E}^{d-r}/d\Omega_{X|E}^{d-r-1})^* \ar[r] & \scriptstyle \varinjlim\limits_E H^d_x(X_{\et}, \Omega_{X|E}^{d-r})^* \ar[r] & \scriptstyle \varinjlim\limits_EH^d_x(X_{\et}, \Omega_{X|E,\log}^{d-r})^{\ast} \ar[r] &0
}\]
where $A^*\coloneqq \mathrm{Hom}_{\Z/p\Z}(A,\Z/p\Z)$ for an abelian group $A$, $\Omega_{X|E}^{d-r}\coloneqq \Omega_X^{d-r}(\log E_{\mathrm{red}})\otimes \mathcal{O}_X(-E)$, and $d\Omega_{X|E}^{d-r-1}\coloneqq \mathrm{Image}(d\colon \Omega_{X|E}^{d-r-1} \to \Omega_X^{d-r})$, and $Z\Omega_{U}^{r}\coloneqq \mathrm{Ker}(d\colon \Omega_{U}^{r} \to \Omega_{U}^{r+1})$.

The proof is same as the proof in \cite{jszduality} and \cite{zhaoduality}, we quickly recall the argument: since $j\colon U\to X$ is affine, we may rewrite $H^0(U, \Omega_{U}^{r})$ as $\varinjlim_EH^0(X, \Omega^i_X(\log E_{\mathrm{red}})\otimes\mathcal{O}_X(E))$. Then we use Theorem \ref{grothendieck.local} for sheaves $\Omega^i_X(\log E_{\mathrm{red}})(-E)$ to conclude that the second and the third vertical arrows are isomorphisms. Hence the assertion follows.
\end{proof}

For $r=0, i=1$, we get
\[  H^1(U_{\et},\Z/p^m\Z) \cong \varinjlim_E\Hom(H^{d}_x(X_{\et}, \DRWlog m {d} {X|E}), \Z/p^m\Z).\]

Similar to Corollary \ref{transition.surjective}, the transition maps are surjective in the projective system, for our divisor $D$ we define

\[ \Fil_D H^1(U_{\et},\Z/p^m\Z) :=\Hom(H^{d}_x(X_{\et}, \cRF), \Z/p^m\Z);\]
by Pontryagin duality, we also define
\[   \pi^{\text{ab}}_1(X,D)/p^m\coloneqq \Hom(\Fil_D H^1(U_{\et},\Z/p^m\Z), \Z/p^m\Z).\]

Theorem \ref{ramified.duality} gives us an isomorphism
\[  H^{d}_x(X_{\et}, \cRF) \xrightarrow{\cong} \pi^{\text{ab}}_1(X,D)/p^m.\]

\begin{proposition}
We have  \[ \normalfont H^{d}_{x}(X_{\nis}, \cRF)  \cong H^{d}_x(X_{\et}, \cRF).\]
\end{proposition}
\begin{proof}
This is similar to the argument in the proof of  Proposition \ref{E0nis=et}. Only the last step, to claim \[ H^{a}_x(X_{\nis},\cAF) \xrightarrow{\cong} H^{a}_x(X_{\et}, \cAF)\]
is an isomorphism, uses different results. In this case, it is an isomorphism since both are isomorphic to $K^M_{d-a}(k(x))/p^m=\DRWlog{m}{d-a}{x}$ by purity \cite[Theorem 3.2]{shihopurity} and the known Gersten conjecture \cite{kerzgersten}.

\end{proof}

\subsection{Class field theory via ideles}

For a complete regular local ring $A$ of dimension $d$ of characteristic $p>0$, and $X, X', U, D, D'$ as before.  An idelic description of $H^d_x(X_{\nis}, \mathcal{K}^M_{d, X|D}) $ is given by the following theorem.
\begin{theorem}{(\cite[Theorem 8.2]{kerzideles})}
There exists an isomorphism
\[ \normalfont C(X', D') \cong H^d_x(X_{\nis}, \mathcal{K}^M_{d, X|D}), \]
\end{theorem}

In summary, the class field theory of henselian regular local ring over $\mathbb{F}_p$ can be reformulated as follows:
\begin{corollary} There is a canonical isomorphism
\[ C(X',D')/p^m  \xrightarrow{\cong} \pi^{\text{ab}}_1(X,D)/p^m.\]
\end{corollary}

\begin{remark} The case $d=2$ has been studied in \cite{saitoCFT2dim}. The case $d=3$ has
  been investigated   in \cite{matsumi} using a slightly different class group. The $\ell$-primary analog has been studied by Sato in \cite{satoladiccft}.
\end{remark}

\section{Class field theory for schemes over discrete valuation rings} \label{sec:cftschemeslocal}
Let $R$ be a henselian discrete valuation ring with fraction field $K$, and let $k$ be its residue field of characteristic $p>0$ which we assume to be finite. We fix an uniformizer $\pi$ of $R$. We use the notation as in the following diagram:
$$\xymatrix@C=3pc@R=4pc {
	X_s \ar@{^{(}->}[r]^-{i} \ar[d]^{f_s} & X \ar[d]^f &X_{\eta}\ar[d]^{f_{\eta}}\ar@{_{(}->}[l]_-j \\
	s=\Spec(k) \ar@{^{(}->}[r]^-{i_s} &B=\Spec(R)&\eta=\Spec(K) \ar@{_{(}->}[l]_-{j_{\eta}}
}
$$
where $f$ is a flat projective of fibre dimension $d$. We assume that $X$ is a regular
scheme with smooth generic fiber  $X_{\eta}$ such that the reduced special fibre
$X_{s,\mathrm{red}}$  is a simple normal crossing divisor.  Let $\jmath\colon
U\hookrightarrow X$ be an open subscheme contained in the generic fibre such that $X\setminus U$ is the support of a
simple normal crossing divisor $D$.

\subsection{Idele class group}
We want to give an idelic description of the class group $H^{d+1}_{X_s}(X_{\nis}, \mathcal{K}^M_{d, X|D})$. We use the dimension function $d(x)=\mathrm{dim}(\overline{\{x\}})$ on $X$.
\begin{definition}
	\begin{itemize}
		\item[(i)] A $Q^{o}$-chain on ($U\subset X$) is a $Q$-chain $P=(p_0,\cdots, p_{s-2}, p_s)$ on ($U\subset X$) such that $s\geq 2$. We denote the set of $Q^{o}$-chain on ($U\subset X$) by $\mathcal{Q}^o$.
		\item[(ii)]The idele class group $C(U\subset X;X_s)$ is
		\[ C(U\subset X;X_s)\coloneqq\mathrm{Coker}(\bigoplus_{P\in
			\mathcal{Q}^{o}}K^M_{d(P)}(k(P))\oplus \bigoplus_{y\in U_\eta^{d-1}} K^M_2(k(y)) \xrightarrow{Q} I(U\subset X)); \]
		\item[(iii)] The idele class group $C(X,D;X_s)$ of $X$ relative to the effective divisor $D$ is defined as
		\[ C(X,D;X_s)\coloneqq\mathrm{Coker}(\bigoplus_{P\in \mathcal{Q}^{o}}K^M_{d(P)}(k(P)) \oplus \bigoplus_{y\in U_\eta^{d-1}} K^M_2(k(y)) \xrightarrow{Q} I(X,D)). \]
	\end{itemize}
\end{definition}

\begin{theorem}\label{idele.cft}\mbox{}
	\begin{itemize}
		\item[(i)]
		There exists a canonical isomorphism
		\[  \normalfont C(X,D;X_s) \cong H^{d+1}_{X_s}(X_{\nis}, \mathcal{K}^M_{d+1, X|D}). \]
		\item[(ii)] $\normalfont H^{d+1}(X_{\nis}, \mathcal{K}^M_{d+1, X|D})=0.$
	\end{itemize}
\end{theorem}
\begin{proof}\renewcommand{\qedsymbol}{}
	Let $\cF$ be the Nisnevich sheaf $\mathcal{K}^M_{d+1, X|D}$. We start with part (i).
	We have seen that the degeneration of the coniveau spectral sequence
	\[  E^{p,q}_{1,\nis}(\cF):=\bigoplus_{x\in X^p} H^{p+q}_x(X_{\nis}, \cF)\Longrightarrow H^{p+q}(X_{\nis}, \cF) \]
	implies
	\begin{equation}\label{coker.milnorK}  H^{d+1}_{X_s}(X_{\nis}, \cF) =\mathrm{Coker}(\bigoplus_{x\in X_{1}\cap X_s}H^{d}_x(X_{\nis}, \cF)\to \bigoplus_{x\in X_0}H^{d+1}_x(X_{\nis}, \cF)).
	\end{equation}
	By definition and \cite[Theorem 8.2]{kerzideles} we obtain an isomorphism
	\begin{equation}\label{coker.classgroup} C(X,D;X_s)\cong \mathrm{Coker}( \bigoplus_{y\in U_\eta^{d-1}} K^M_2(k(y)) \to \bigoplus_{x\in X_0}H^{d+1}_x(X_{\nis}, \cF))
	\end{equation}
	It is sufficient to observe that the canonical map
	\[
	\bigoplus_{y\in U_\eta^{d-1}} K^M_2(k(y)) \to \bigoplus_{x\in X_{1}\cap X_s}H^{d}_x(X_{\nis}, \cF)
	\]
	is surjective, see~\cite[Sec.~6]{kerzideles}. This finishes the proof of part (i).

	For part (ii) we use the isomorphism
	\[
	H^{d+1}(X_{\nis}, \cF) =\mathrm{Coker}(\bigoplus_{x\in X_{1}}H^{d}_x(X_{\nis}, \cF)\to
	\bigoplus_{x\in X_0}H^{d+1}_x(X_{\nis}, \cF))
	\]
	and the surjectivity of
	\[
	\bigoplus_{x\in X_{1}\cap X_\eta} K^M_1(k(x))\to
	\bigoplus_{x\in X_0}H^{d+1}_x(X_{\nis}, \cF),
	\]
	see~\cite[Sec.~6]{kerzideles}.
\end{proof}

Note that the generic fiber $X_{\eta}$ is a smooth variety over the local field $K$.
Its class field theory has been  studied in several cases, for example the case  $d=1$  is
well understood  by work of Bloch and Saito, see \cite{saito85cft} and \cite{hiranouchi}.
In \cite{forre15}, Forr\'e determines the kernel of the reciprocity map in unramified
$\ell$-adic class field theory in the higher dimension case.

\begin{definition}
	Assume $\mathrm{Supp}(D) \supset X_s$, we denote $D_{\eta}=D\times_X X_{\eta}$, and define
	\[  \normalfont \widehat{SK}_1(U):=\varprojlim\limits_D C(X,D;X_s)=
	\varprojlim\limits_EH^{d+1}_{X_s}(X_{\nis}, \mathcal K^M_{d+1, X|E}), \]
	where the limit is over all effective divisors $E$ with support $X\setminus U$.
	\[  SK_1(X_\eta, D_\eta):=H^d(X_{\eta,\nis}, \cRM{d+1}).\]
\end{definition}
\begin{remark}
	\begin{itemize}
		\item[(i)]We have seen that, by the degeneration of the coniveau spectral sequence, the group $\normalfont SK_1(X_\eta, D_\eta)=H^d(X_{\eta,\nis}, \cRM{d+1})$ is isomorphic to
		\begin{equation}\label{SK_coker}
		\normalfont \mathrm{coker}(\bigoplus_{y\in (X_{\eta})_1}H^{d-1}_y(X_{\eta,\nis}, \cRM{d+1})\xrightarrow{\partial} \bigoplus_{x\in (X_{\eta})_0}H^{d}_x(X_{\eta,\nis}, \cRM{d+1})).
		\end{equation}
		Using the methods from~\cite{kerzideles} it is easy to write down an idelic description of
		this group, for example if $D_{\eta}=0$ then $SK_1(X_{\eta},0)=SK_1(X_{\eta})$ where $SK_1(X_{\eta})$ is defined as
		\[ \mathrm{coker}(\bigoplus_{y\in (X_{\eta})_1}K^M_2(\kappa(y))\xrightarrow{\partial} \bigoplus_{x\in (X_{\eta})_0}\kappa(x)^{\times}). \]
		\item[(ii)] If $d=1$ and $\mathrm{Supp}(D)=X_s$, then
		$\widehat{SK}_1(U)=\widehat{SK}_1(X_{\eta})$, which has been defined in
		\cite{katosaitounramified} via the idelic method.
		\item[(iii)] By Theorem~\ref{idele.cft} we get a canonical surjection
		\[
		SK_1(X_\eta, D_\eta) \to C(X,D;X_s).
		\]
		We do not know, whether this map is an isomorphism in general, but
		Theorem~\ref{compare.l.adic.SK} suggests that it is so at least after tensoring with
		$\mathbb Z / n \mathbb Z$ for any integer $n>0$.
	\end{itemize}
\end{remark}

\subsection{Kato complexes on simple normal crossing varieties} \label{katoconj_snc}
We recall notations and theorems in \cite{jskatohomology}.
Let $Y$ be a proper simple normal crossing variety over the finite field $k$ of dimension $d$, and let $Y_1,\cdots, Y_N$ be its smooth irreducible components. Let
\[ Y_{i_1,\cdots,i_s}:=Y_{i_1}\times_Y\cdots \times_YY_{i_s} \]
be the scheme-theoretic intersection of $Y_{i_1},\cdots, Y_{i_s}$, and denote
\[  Y^{[s]}:= \coprod_{1\leq i_1 < \cdots < i_s\leq N}Y_{i_1,\cdots,i_s} \]
for the disjoint union of the $s$-fold intersections of the $Y_i$, for any $s>0$. Since $Y$ is simple, all  $Y^{[s]}$ are smooth of dimension $d-s+1$. The immersions $Y_{i_1,\cdots,i_s} \hookrightarrow Y$ and $Y_{i_1,\cdots,i_s}\hookrightarrow Y_{i_1,\cdots,\hat{i_v},\cdots,i_s}$ induce canonical maps $$i^{[s]} \colon Y^{[s]} \to Y,\quad \delta_{\nu} \colon Y^{[s]} \to Y^{[s-1]}.$$

For integer $n>0, i\geq 0$ we define the following \'etale sheaves on $Y$:
\begin{itemize}
	\item[(i)] If $p\nmid n$, then let $\Z/n\Z(i)\coloneqq \mu^{\otimes i}_{n,Y}$ be the $i$-th tensor power over $\Z/n\Z$ of the sheaf of $n$-th roots of unity.
	\item[(ii)] If $n=mp^r, r\geq 0$ with $p\nmid m$, then let
	\[  \Z/n\Z(i)\coloneqq\DRWlognc{r}{i}{Y}[-i]\oplus\mu^{\otimes i}_{m,Y}  \]
	where $\DRWlognc{r}{i}{Y}(U)\coloneqq\ker(\partial: \bigoplus_{x\in
		U^0}W_r\Omega_{x,\log}^i \to \bigoplus_{x\in U^1}W_r\Omega_{x,\log}^{i-1})$ for
	$U\subset Y$ open.
	Note that $\DRWlognc{r}{d}{Y}=\DRWlog{r}{d}{Y}$ if $Y$ is smooth \cite[1.3.2]{satoncv}.
\end{itemize}
The Kato complex $C^{1,0}(Y,\Z/n\Z(d))$ is defined to be the  complex:
\begin{multline*}
\bigoplus_{y\in Y^0}H^{d+1}(y, \Z/n\Z(d)) \to \bigoplus_{y\in Y^1}H^{d}(y,\Z/n\Z(d-1)) \to \cdots \\
\cdots \to \bigoplus_{y\in Y^a}H^{d-a+1}(y, \Z/n\Z(d-a))\to \cdots \to \bigoplus_{y\in Y^d} H^{1}(y, \Z/n\Z) ,
\end{multline*}
where $\Z/n\Z(i)$ is defined as above for the residue field of $Y$ at $y$, and put the term $\bigoplus_{y\in Y^a}$ in degree $a-d$ as an object in derived category. Similarly, for each $s$, on $Y^{[s]}$ we define  the complex  $C^{1,0}(Y^{[s]}, \Z/n\Z(d-s+1))$, and moreover we define the complex $C({Y}^{\bullet}, \Z/n\Z)$ as
\[ \cdots \to (\Z/n\Z)^{\pi_0(Y^{[s+1]})} \xrightarrow{d_s}  (\Z/n\Z)^{\pi_0(Y^{[s]})} \cdots \to (\Z/n\Z)^{\pi_0(Y^{[1]})}, \]
where $\pi_0(Z)$ is the set of connected components of a scheme $Z$, the last term of this complex is placed in degree $0$, and the differential $d_s$ is $\sum_{\nu=1}^{s+1}(-1)^{\nu+1}(\delta_{\nu})_*$.
\begin{theorem}{(\cite[Proposition 3.6 and Theorem 3.9]{jskatohomology})}\label{katoconjsncv}
	\begin{itemize}
		\item[(i)]There is a spectral sequence \[  E^1_{s,t}({Y}^{\bullet},\Z/n\Z)=H_t(C^{1,0}(Y^{[s+1]},\Z/n\Z(d-s)) \Rightarrow H_{s+t}(C^{1,0}(Y, \Z/n\Z(d)))  \]
		in which the differentials $d^1_{s,t}=\sum_{\nu=1}^{s+1}(-1)^{\nu+1}(\delta_{\nu})_*$.
		\item[(ii)] We have  $E^1_{s,t}({Y},\Z/n\Z)=0$ if $t<0$, and hence there are canonical edge morphisms
		\[ e_a^{\mathcal{Y},p^m} \colon H_a(C^{1,0}(Y, \Z/n\Z(d))) \to E^2_{a,0}({Y}^{\bullet},\Z/n\Z).  \]
		\item[(iii)] The trace map induces a canonical isomorphism
		\[ \mathrm{tr}\colon E^2_{a,0}({Y}^{\bullet},\Z/n\Z) \to H_a(C({Y}^{\bullet},\Z/n\Z)); \]
		\item[(iv)] The composite of edge and trace morphisms gives us a canonical map
		\[  \gamma^{Y,p^m}_a \colon H_a(C^{1,0}(Y, \Z/n\Z(d))) \to H_a(C(Y^{\bullet},\Z/n\Z)),\]
		which is an isomorphism if $0\leq a \leq 4$.
	\end{itemize}
\end{theorem}
\begin{remark}
	In the following, we need the cases $a=1$ and $a=2$, which will give us an explicit description of $E_2$-terms of certain coniveau spectral sequences.
\end{remark}
\subsection{The $\ell$-primary part}

Let $\ell$ be a prime number and $\ell\neq p$.
The cup product induces the following morphism
\[  R\jmath_*\mu_{\ell^m,U}^{\otimes r} \to R\jmath_*\cHom_U(\mu_{\ell^m,U}^{\otimes d+1-r}, \mu_{\ell^m,U}^{\otimes d+1}). \]
As $\mu_{\ell^m,U}^{\otimes d+1}=\jmath^*\mu_{\ell^m,X}^{\otimes d+1}$ the adjoint pair
$(\jmath_!,\jmath^*)$ gives an isomorphism
\[  R\jmath_*R\cHom_U(\mu_{\ell^m,U}^{\otimes d+1-r}, \mu_{\ell^m,U}^{\otimes d+1})=R\cHom_X(\jmath_!\mu_{\ell^m,U}^{\otimes d+1-r}, \mu_{\ell^m,X}^{\otimes d+1}). \]
Using the adjoint pair $(i_*, Ri^!)$ and these two maps above, we obtain a pairing on $X_{\et}$:
\begin{equation}
i^*R\jmath_*\mu_{\ell^m,U}^{\otimes r}\otimes^L Ri^!\jmath_!\mu_{\ell^m,U}^{\otimes d+1-r} \to Ri^!\mu_{\ell^m,X}^{\otimes d+1}.
\end{equation}
Therefore a pairing of cohomology groups:
\begin{equation}\label{pairing.l.adic}
H^i(U_{\et}, \mu_{\ell^{m},U}^{\otimes r}) \times H^{j}_{X_s}(X_{\et}, \jmath_!\mu_{\ell^{m},U}^{\otimes d+1-r}) \to H^{i+j}_{X_s}(X_{\et}, \mu_{\ell^{m},X}^{\otimes d+1}).
\end{equation}

We have the following duality theorem, see~\cite[Thm.~7.5]{geisserduality}.
\begin{theorem}
	\begin{itemize}
		\item[(i)] There is a canonical isomorphism, so called the trace map,
		\[\normalfont \mathrm{Tr}\colon H^{2d+3}_{X_s}(X_{\et}, \mu_{\ell^{m},X}^{\otimes d+1}) \xrightarrow{\cong}  \Z/\ell^m\Z  \]
		\item[(ii)] The trace map Tr and the pair (\ref{pairing.l.adic}) induce a perfect pairing of finite groups
		\[ \normalfont H^i(U_{\et}, \mu_{\ell^{m},U}^{\otimes r}) \times H^{2d+3-i}_{X_s}(X_{\et}, \jmath_!\mu_{\ell^{m},U}^{\otimes d+1-r}) \to H^{2d+3}_{X_s}(X_{\et}, \mu_{\ell^{m},X}^{\otimes d+1})  \xrightarrow{\mathrm{Tr}} \Z/\ell^m\Z\]
	\end{itemize}
\end{theorem}

For $r=0, i=1$, we obtain
\[  H^1(U_{\et},\Zlm) \cong \Hom(H^{2d+2}_{X_s}(X_{\et}, \jmath_!\mu_{\ell^{m}}^{\otimes d+1}), \Zlm),\]
and by Pontryagin duality
\begin{equation}\label{et.fund=et.coh}
H^{2d+2}_{X_s}(X_{\et}, \jmath_!\mu_{\ell^{m}}^{\otimes d+1})\cong \pi^{\text{ab}}_1(U)/\ell^m. \end{equation}

For any abelian sheaf $\cF$ on $X_{\nis}$ or $X_{\et}$, we have the following two coniveau spectral sequences:
\[  E^{p,q}_{1,\et}(\cF):=\bigoplus_{x\in X^p\cap X_s} H^{p+q}_x(X_{\et}, \cF)\Longrightarrow H^{p+q}_{X_s}(X_{\et}, \cF), \]
\[  E^{p,q}_{1,\nis}(\cF):=\bigoplus_{x\in X^p \cap X_s} H^{p+q}_x(X_{\nis}, \cF)\Longrightarrow H^{p+q}_{X_s}(X_{\nis}, \cF) .\]
\begin{proposition}\label{coniveau.l.adic}
	\mbox{}
	\begin{itemize} 
		\item [(i)]  $\normalfont E^{\bullet,d+2}_{1,\et}(\jmath_!\mu_{\ell^m,U}^{\otimes d+1}) \cong E^{\bullet,d+2}_{1,\et}(\mu_{\ell^m,X}^{\otimes d+1})$.
		\item [(ii)] The local Chern class map induces a surjection
		$ \normalfont E^{\bullet,0}_{1,\nis}(\mathcal{K}^M_{d+1, X|D})/\ell^m
		\twoheadrightarrow
		E^{\bullet,d+1}_{1,\et}(\jmath_!\mu_{\ell^m,U}^{\otimes d+1})$ and an isomorphism $\normalfont E^{d+1,0}_{1,\nis}(\mathcal{K}^M_{d+1, X|D})/\ell^m \cong E^{d+1,d+1}_{1,\et}(\jmath_!\mu_{\ell^m,U}^{\otimes d+1}).$
	\end{itemize}
\end{proposition}
\begin{proof}
	The argument is analogous to that that in Section \ref{l-part.over.finite fields}.
	More precisely, part $(i)$ corresponds to Proposition~\ref{ladic.top} and part
	(ii) corresponds to Theorem~\ref{Sato.ladic}.
\end{proof}
\begin{corollary}
	There are canonical isomorphisms\[ \normalfont H^{d+1}_{X_s}(X_{\nis}, \mathcal{K}^M_{d+1,X|D})/\ell^m\cong E^{d+1,0}_{2,\nis}(\mathcal{K}^M_{d+1, X|D})/\ell^m \cong E^{d+1,d+1}_{2,\et}(\jmath_!\mu_{\ell^m,U}^{\otimes d+1}).\]
\end{corollary}
\begin{proof}
	The degenerating coniveau spectral sequence on $X_{\nis}$ gives the first isomorphism. The
	second isomorphism results from the same argument as in Corollary~\ref{ladic.nis=et} using Proposition~\ref{coniveau.l.adic}(ii).
\end{proof}
By purity
the complex $E^{\bullet,d+2}_{1,\et}(\mu_{\ell^m,X}^{\otimes d+1})$ is isomorphic
to the  Kato complex $C^{1,0}(X_s, \Zlm(d))$ from Subsection \ref{katoconj_snc} (up to a shift), i.e.~to

\begin{multline*}
\bigoplus_{y\in X_s^0}H^{d+1}(y, \Zlm(d)) \to \bigoplus_{y\in X_s^1}H^{d}(y,\Zlm(d-1)) \to \cdots \\
\cdots \to \bigoplus_{y\in X_s^a}H^{d-a+1}(y, \Zlm(d-a))\to \cdots \to \bigoplus_{y\in X_s^d} H^{1}(y, \Zlm) ,
\end{multline*}
where we set the last term in degree $0$ as an object in the derived category.

\begin{theorem}\label{nis.class.to.et.class}
	The canonical morphism
	\[ \normalfont H^{d+1}_{X_s}(X_{\nis}, \mathcal{K}^M_{d+1, X|D})/\ell^m \to H^{2d+2}_{X_s}(X_{\et},  \jmath_!\mu_{\ell^{m}}^{\otimes d+1})\]
	fits into an exact sequence
	\begin{equation}
	\footnotesize \normalfont H_2(C(X_s^{\bullet},\Z/\ell^m\Z)) \to  H^{d+1}_{X_s}(X_{\nis}, \mathcal{K}^M_{d+1, X|D})/\ell^m  \to  H^{2d+2}_{X_s}(X_{\et},  \jmath_!\mu_{\ell^{m}}^{\otimes d+1})\to  H_1(C(X_s^{\bullet},\Z/\ell^m\Z)) \to 0.
	\end{equation}
\end{theorem}

\begin{proof}
	By the coniveau spectral sequence for $\cF=\jmath_!\mu_{\ell^m,U}^{\otimes d+1}$ on $X_{\et}$, we have an exact sequence:
	\[ E_{2,\et}^{d-1,d+2}(\cF) \to E_{2,\et}^{d+1,d+1}(\cF) \to H^{2d+2}_{X_s}(X_{\et}, \cF) \to E_{2,\et}^{d,d+2}(\cF) \to 0. \]
	Using Proposition \ref{coniveau.l.adic}, we have
	\[ E_{2,\et}^{d+1,d+1}(\cF) =E_{2,\nis}^{d+1,0}(\mathcal{K}^M_{d+1, X|D}/\ell^m )=H^{d+1}_{X_s}(X_{\nis}, \mathcal{K}^M_{d+1, X|D}/\ell^m)=H^{d+1}_{X_s}(X_{\nis}, \mathcal{K}^M_{d+1, X|D})/\ell^m.\]
	Moreover combining with Theorem \ref{katoconjsncv}, we obtain
	\[ E_{2,\et}^{d-1,d+2}(\cF)=E_{2,\et}^{d-1,d+2}(\mu_{\ell^m,X}^{\otimes d+1}) = H_2(C(X_s^{\bullet},\Zlm));   \]
	\[ E_{2,\et}^{d,d+2}(\cF)=E_{2,\et}^{d,d+2}(\mu_{\ell^m,X}^{\otimes d+1}) = H_1(C(X_s^{\bullet},\Zlm)).   \]
\end{proof}
In summary, combining Theorem \ref{nis.class.to.et.class} and Theorem \ref{idele.cft} with the identification (\ref{et.fund=et.coh}), we reformulate the $\ell$-primary part of class field theory in this setting as follows.
\begin{theorem}
	There is a canonical map
	\[ \rho_{X,D}\colon C(X,D;X_s)/\ell^m \to \pi^{\mathrm{ab}}_1(U)/\ell^m, \]
	which fits into an exact sequence of finite groups
	\[ H_2(C(X_s^{\bullet},\Zlm)) \to C(X,D;X_s)/\ell^m \to  \pi^{\mathrm{ab}}_1(U)/\ell^m \to H_1(C(X_s^{\bullet},\Zlm)) \to 0.\]
	Equivalently, there is an exact sequence:
	\begin{equation}\label{SKhat.to.fund}
	H_2(C(X_s^{\bullet},\Zlm)) \to \widehat{SK}_1(U)/\ell^m \to  \pi^{\mathrm{ab}}_1(U)/\ell^m \to H_1(C(X_s^{\bullet},\Zlm)) \to 0.\end{equation}
\end{theorem}
\begin{proof}\renewcommand{\qedsymbol}{}
	The map is defined by the following diagram
	\[ \xymatrix{ C(X,D;X_s)/\ell^m \ar[r]^-{\cong} \ar@{-->}[d]_{\rho_{X,D}}& H^{d+1}_{X_s}(X_{\nis}, \mathcal{K}^M_{d+1, X|D})/\ell^m  \ar[d]\\
		\pi^{\mathrm{ab}}_1(U)/\ell^m & H^{2d+2}_{X_s}(X_{\et}, \jmath_!\mu_{\ell^{m}}^{\otimes d+1}). \ar[l]_-{\cong}
	}\]
	So the first exact sequence is a direct consequence of Theorem~\ref{nis.class.to.et.class}.
	The second exact sequence results from the fact that
	\begin{equation}\label{SKhat.mod.l}
	\widehat{SK}_1(U)/\ell^m =H^{d+1}_{X_s}(X_{\nis}, \mathcal{K}^M_{d+1, X|D})/\ell^m
	\end{equation}
	for any $D$ with $\Supp(D)=X\setminus U$. Indeed, we denote $D_0=X\setminus U$ the reduced divisor, it suffices to show the following claim.\end{proof}
\begin{claim} We have
	\begin{equation*}
\normalfont	\big(\varprojlim\limits_DH^{d+1}_{X_s}(X_{\nis}, \cRM{d+1})\big) \otimes_{\Z}\Z/\ell^m\Z \xrightarrow{\cong} H^{d+1}_{X_s}(X_{\nis}, \mathcal{K}^M_{d+1,X|D_0})/\ell^m.
	\end{equation*}
\end{claim}
\begin{proof}[Proof of Claim]
	The canonical surjective map \[ \varphi_{D}\colon H^{d+1}_{X_s}(X_{\nis}, \cRM{d+1}) \to H^{d+1}_{X_s}(X_{\nis}, \mathcal{K}^M_{d+1,X|D_0})\]
	fits into the exact sequence
	\begin{equation}\label{ker.mod.l}
	\xymatrix{  0\ar[r] &\ker(\varphi_D) \ar[r]&H^{d+1}_{X_s}(X_{\nis}, \cRM{d+1}) \ar[r]^-{\varphi_D} &H^{d+1}_{X_s}(X_{\nis}, \mathcal{K}^M_{d+1,X|D_0}) \ar[r]&0
	}
	\end{equation}
	Applying $\varprojlim_D$ to the above exact sequence, we obtain an exact sequence
	\begin{equation}\label{exact.limit.l}
	\xymatrix@C=0.5cm{  0\ar[r]& \varprojlim\limits_D\ker(\varphi_D) \ar[r]&\varprojlim\limits_DH^{d+1}_{X_s}(X_{\nis}, \cRM{d+1}) \ar[r] &H^{d+1}_{X_s}(X_{\nis}, \mathcal{K}^M_{d+1,X|D_0}) \ar[r]&0.
	}
	\end{equation}

	By the long exact sequence associated to the short exact sequence
	$$0\to \cRM{d+1} \to  \mathcal{K}^M_{d+1,X|D_0} \to \mathcal{K}^M_{d+1,X|D_0}/\cRM{d+1}\to 0,$$ we see that $H^d_{X_s}(X_{\nis},\mathcal{K}^M_{d+1,X|D_0}/\cRM{d+1} ) \twoheadrightarrow \ker(\varphi_D) $ is surjective. Proposition \ref{gr.log.form}(ii) tells us that $H^d_{X_s}(X_{\nis},\cRM{d+1}/\mathcal{K}^M_{d+1,X|D_0} )$ is $p$-primary torsion group, therefore  in particular $\ker(\varphi_D)$ is a $\Z_{(p)}$-module, so is the inverse limit $\varprojlim_D\ker(\varphi_D)$. It follows that
	\[ \Z/\ell^m\Z\otimes_{\Z}\varprojlim_D\ker(\varphi_D)=0.\]
	Tensoring the exact sequence (\ref{exact.limit.l}) with $\Z/\ell^m\Z$, we obtain the claim.
\end{proof}

In the case that $\Supp(D)=X_s$, we have the following diagram:
\[\xymatrix{ & SK_1(X_{\eta})/\ell^m \ar[d]^{\phi} \ar[r]^{\rho_{X_{\eta}}}&\pi^{\mathrm{ab}}_1(X_{\eta})/\ell^m \ar@{=}[d] &&\\
	H_2(C(X_s^{\bullet},\Zlm)) \ar[r] &\widehat{SK}_1(X_{\eta})/\ell^m \ar[r]^{\rho_{X,X_s}}&  \pi^{\mathrm{ab}}_1(X_{\eta})/\ell^m \ar[r]& H_1(C(X_s^{\bullet},\Zlm)) \ar[r] &0,
}\]
where the last row is the exact sequence (\ref{SKhat.to.fund}), the morphism $\rho_{X_{\eta}}$ is the reciprocity map of variety over the local field $K$ (cf.\cite{katosaitounramified}), and the  map $\phi$ is induced by the connection map $H^{d}(X_{\eta}, \mathcal{K}^M_{d+1, X_{\eta}}) \to H^{d+1}_{X_s}(X_{\nis}, \mathcal{K}^M_{d+1, X|D})$.

In the remainder of this subsection, we explain why our new approach recovers the known result for varieties over local fields (cf.\cite{forre15}) in the good reduction case.
\begin{theorem}\label{compare.l.adic.SK}
	If $\Supp(D)=X_s$ is smooth, then the map $\phi\colon SK_1(X_{\eta})/\ell^m \to \widehat{SK}_1(X_{\eta})/\ell^m$ is an isomorphism.
\end{theorem}
To proof this theorem, we may further assume that $D=X_s$, since the multiplicity of $D$
has no contribution to $\widehat{SK}_1(X_{\eta})/\ell^m$.  To simplify our notations, we
denote $\Lambda(i)_Y\coloneqq\Z/\ell^m\Z\otimes \Z(i)_Y$ for a scheme $Y$ and $i\in \Z$,
where $\Z(i)$ is Bloch's cycle complex on the small Nisnevich site
(cf. \cite{geissermotivic}).

We can define the restriction map $r_i\colon \Lambda(i)_X\to i_*\Lambda(i)_{X_s}$ as the composition
\[ \Lambda(i)_{X}\to j_*\Lambda(i)_{X_{\eta}} \xrightarrow{\cdot \pi} j_*\Lambda(i+1)_{X_{\eta}}[1] \to  i_*\Lambda(i)_{X_s}, \]
where the middle arrow is given by multiplication by $\pi$, and the last arrow is the localization map.

Let
\[ \Lambda(i)_{X|X_s}\coloneqq \mathrm{hofib}(r_i\colon \Lambda(i)_X\to i_*\Lambda(i)_{X_s}) \]
be the homotopy fiber of $r_i$.
By rigidity~\cite[Thm.~1.2.(3)]{geissermotivic} we get an isomorphism $j_! \Lambda
(i)_{X_\eta} \cong \Lambda(i)_{X|X_s}$. Notice that we also have an analogous isomorphism
$j_! \mathcal K^M_{i,X_\eta}/\ell^m \cong  \mathcal K^M_{i,X|X_s}/\ell^m$. So we conclude:

\begin{proposition} \label{motivic.class.group}
	There is a canonical isomorphism $$\mathcal K^M_{i,X|X_s}/\ell^m \cong \mathcal
	H^i(\Lambda(i)_{X|X_s})$$ and $ \mathcal H^j(\Lambda(i)_{X|X_s})=0 $ for $j>i$.
\end{proposition}

Note that Proposition~\ref{motivic.class.group} implies that the canonical map
\begin{equation}\label{iso.relmot}
H^{2d+2}_{X_s}(X_{\nis}, \Lambda(d+1)_{X|X_s}) \xrightarrow{\cong} H^{d+1}_{X_s}(X_{\nis}, \mathcal{K}^M_{d+1, X|X_s})/\ell^m
\end{equation}
is an isomorphism.

To finish the proof of Theorem \ref{compare.l.adic.SK}, we also need the following result:
\begin{proposition}\label{vanishing.motivic.2d+1}
	The group $\normalfont H^{2d+1}(X_{\nis}, \Lambda(d+1)_{X|X_s})=0$.
\end{proposition}
\begin{proof}
	By the definition of $\Lambda(d+1)_{X|X_s}$, there is a long exact sequence
	\begin{multline*}
	H^{2d}(X_{\nis}, \Lambda(d+1)_{X}) \xrightarrow{\alpha} H^{2d}(X_{s,\nis}, \Lambda(d+1)_{X_s}) \to  H^{2d+1}(X_{\nis}, \Lambda(d+1)_{X|X_s})\\ \to H^{2d+1}(X_{\nis}, \Lambda(d+1)_{X}) \xrightarrow{\beta} H^{2d+1}(X_{s,\nis}, \Lambda(d+1)_{X_s}).
	\end{multline*}
	It suffices to show that $\alpha$ is surjective and $\beta$ is injective. In fact, using the relation between motivic cohomology and higher Chow groups, we will show that both $\alpha$ and $\beta$ are isomorphisms. More precisely, the fact that $\alpha$ is an isomorphism follows from the diagram:
	\[ \xymatrix{ H^{2d}(X_{\nis},\Lambda(d+1)_{X}) \ar@{=}[r]\ar[d]^-{\alpha} &CH^{d+1}(X,2;\Z/\ell^m\Z) \ar[r]^-{\cong} & H^{2d}(X_{\et}, \mu_{\ell^m}^{\otimes d+1}) \ar[d]^-{\cong}\\
		H^{2d}(X_{s,\nis}, \Lambda(d+1)_{X_s})\ar@{=}[r] &CH^{d+1}(X_s,2;\Z/\ell^m\Z) \ar[r]^-{\cong} & H^{2d}(X_{s,\et}, \mu_{\ell^m}^{\otimes d+1}),
	}\]
	where the equalities in the rows are the definitions of higher Chow groups with coefficients in $\Z/\ell^m\Z$ (cf. \cite{geisserlevineBlochKato}), the two horizontal arrows are isomorphisms by the known Kato conjecture (\cite[Theorem 9.3]{kerzsaitoIHES}), and the right vertical is the proper base change theorem (SGA4$\frac{1}{2}$,\cite[Arcata IV]{SGA41/2}). The assertion for $\beta$ are similar:
	\[ \xymatrix{ H^{2d+1}(X_{\nis},\Lambda(d+1)_{X}) \ar@{=}[r]\ar[d]^-{\beta} &CH^{d+1}(X,1;\Z/\ell^m\Z) \ar[r]^-{\cong} & H^{2d+1}(X_{\et}, \mu_{\ell^m}^{\otimes d+1}) \ar[d]^-{\cong}\\
		H^{2d+1}(X_{s,\nis}, \Lambda(d+1)_{X_s})\ar@{=}[r] &CH^{d+1}(X_s,2;\Z/\ell^m\Z) \ar[r]^-{\cong} & H^{2d+1}(X_{s,\et}, \mu_{\ell^m}^{\otimes d+1}).
	}\]
\end{proof}
\begin{proof}[Proof of Theorem \ref{compare.l.adic.SK}]
	The assertion  follows directly from the diagram:
	\[ \xymatrix@C=5mm{
		H^{2d+1}(X_{\nis}, \Lambda(d+1)_{X|X_s}) \ar[r] \ar@{=}[d]^-{Prop. \ref{vanishing.motivic.2d+1}}& H^{2d+1}(X_{\eta,\nis}, \Lambda(d+1)_{X_{\eta}}) \ar[r]\ar[d]^-{\cong}& H^{2d+2}_{X_s}(X_{\nis}, \Lambda(d+1)_{X|X_s}) \ar[r]\ar[d]^-{\cong}&0\\
		0 &SK_1(X_\eta)/\ell^m \ar[r]^{\phi}& \widehat{SK}_1(X_{\eta})/\ell^m &
	}\]
	where the first row is the exact localization sequence, note that
	$j^*\Lambda(d+1)_{X|X_s}=\Lambda(d+1)_{X_{\eta}}$. The first vertical isomorphism is given by~\eqref{iso.relmot}
	and the second vertical isomorphism is given by Proposition \ref{motivic.class.group} and~\eqref{SKhat.mod.l}.
\end{proof}

\yigengrem{}
\subsection{The $p$-primary part: equi-characteristic}
Due to the lack of ramified duality in the mixed characteristic case for $p$-primary sheaves, we only treat the case that $R=\Fq[[t]]$ in this subsection and assume $X_s$ is reduced.
In \cite{zhaoduality}, we proved the following duality theorem for the relative logarithmic de Rham-Witt sheaves in this setting.

\begin{theorem}[{\cite[Theorem 3.4.2]{zhaoduality}}]\label{semistable.ramified.duality} Let $R=\Fq[[t]]$.
	There is a perfect pairing of topological abelian groups
	\[ \normalfont H^i(U_{\et},\DRWlog{m}{r}{U}) \times \varprojlim\limits_{E} H^{d+2-i}_{X_s}(X_{\et}, \DRWlog m {d+1-r} {X|E}) \to H^{d+2}_{X_s}(X_{\et}, \DRWlog m {d+1} X) \xrightarrow{\mathrm{Tr}} \Z/p^m\Z ,\]
	where the inverse limit runs over the set of effective divisors $D$ such that $\mathrm{Supp}(D) \subset X-U$. The first group is endowed with the discrete topology, and the second is with profinite topology.
\end{theorem}

For $r=0, i=1$, we get
\[  H^1(U_{\et},\Z/p^m\Z) \cong \varinjlim_E\Hom(H^{d+1}_{X_s}(X_{\et}, \DRWlog m {d+1} {X|E}), \Z/p^m\Z).\]

Similar to Corollary \ref{transition.surjective}, the transition maps are surjective in the projective limit, for our divisor $D$ we define

\[ \Fil_D H^1(U_{\et},\Z/p^m\Z) :=\Hom(H^{d+1}_{X_s}(X_{\et}, \cRFr{d+1}), \Z/p^m\Z);\]
by Pontryagin duality, we also define
\[   \pi^{\text{ab}}_1(X,D)/p^m:= \Hom(\Fil_D H^1(U_{\et},\Z/p^m\Z), \Z/p^m\Z).\]

Therefore Theorem \ref{semistable.ramified.duality} gives us an isomorphism
\[  H^{d+1}_{X_s}(X_{\et}, \cRFr{d+1}) \xrightarrow{\cong} \pi^{\text{ab}}_1(X,D)/p^m.\]

As before we want to compare the group $H^{d+1}_{X_s}(X_{\nis}, \cRFr{d+1})$ with  $H^{d+1}_{X_s}(X_{\et}, \cRFr{d+1})$, by using the coniveau spectral sequence.

For any abelian sheaf $\cF$ on $X_{\nis}$ or $X_{\et}$, we have the following two coniveau spectral sequences:
\[  E^{p,q}_{1,\et}(\cF):=\bigoplus_{x\in X^p\cap X_s} H^{p+q}_x(X_{\et}, \cF)\Longrightarrow H^{p+q}_{X_s}(X_{\et}, \cF) \]
\[  E^{p,q}_{1,\nis}(\cF):=\bigoplus_{x\in X^p \cap X_s} H^{p+q}_x(X_{\nis}, \cF)\Longrightarrow H^{p+q}_{X_s}(X_{\nis}, \cF) .\]

\begin{proposition}\label{etandnis}
	We have the following isomorphisms:
	\begin{itemize}
		\item [(i)] $\normalfont E^{\bullet,1}_{1,\et}(\cRFr{d+1}) \cong E^{\bullet,1}_{1,\et}(\DRWlog{m}{d+1}{X})$;
		\item [(ii)] $ \normalfont E^{\bullet,0}_{1,\nis}(\cRFr{d+1}) \cong E^{\bullet,0}_{1,\et}(\cRFr{d+1})$.
	\end{itemize}
\end{proposition}
\begin{proof}
	This is a local question. The first claim follows by the same argument as in Proposition
	\ref{E1acyclic}, and the second as in Proposition \ref{E0nis=et}.
\end{proof}

By purity~\cite[Theorem 3.2]{shihopurity}
the complex $E^{\bullet,1}_{1,\et}(\DRWlog{m}{d+1}{X})$ is isomorphic to the Kato complex
$C^{1,0}(X_s, \Zpm(d))$ (up to a shift), i.e.~to
\begin{multline*}
\bigoplus_{y\in X_s^0}H^{d+1}_y(X_{s,\et}, \Zpm(d)) \to \bigoplus_{y\in X_s^1}H^{d+2}_y(X_{s,\et},\Zpm(d)) \to \cdots \\
\cdots \to \bigoplus_{y\in X_s^a}H^{d+a+1}_y(X_{s,\et}, \Zpm(d))\to \cdots \to \bigoplus_{y\in X_s^d} H^{2d+1}_y(X_{s,\et}, \Zpm(d)) ,
\end{multline*}
where $\Zpm(d)=\DRWlognc{m}{d}{X_s}[-d]$ and where the last term is placed in degree $0$.

\begin{theorem}\label{nis.to.et}
	The canonical map  \[ H^{d+1}_{X_s}(X_{\nis}, \cRFr{d+1})  \to  H^{d+1}_{X_s}(X_{\et}, \cRFr{d+1}) \] fits into an exact sequence of finite groups
	\begin{equation}\label{kernel.cokernel.rec}
	\normalfont \footnotesize H_2(C(X_s^{\bullet},\Zpm)) \to \footnotesize H^{d+1}_{X_s}(X_{\nis}, \cRFr{d+1}) \to  \footnotesize H^{d+1}_{X_s}(X_{\et}, \cRFr{d+1}) \to \footnotesize H_1(C(X_s^{\bullet},\Zpm)) \to 0
	\end{equation}
\end{theorem}
\begin{proof}
	By the coniveau spectral sequence for $\cF=\cRFr{d+1}$ on $X_{\et}$, we have the following exact sequence
	\[  E_{2,\et}^{d-1,1}(\cF) \to E_{2,\et}^{d+1,0}(\cF) \to H^{d+1}_{X_s}(X_{\et}, \cF) \to E_{2,\et}^{d,1}(\cF) \to 0.\]
	By Proposition \ref{etandnis}, we have \[ E_{2,\et}^{d+1,0}(\cF) =E_{2,\nis}^{d+1,0}(\cF)=H^{d+1}_{X_s}(X_{\nis}, \cF).\]
	Moreover combining with Theorem \ref{katoconjsncv}, we obtain
	\[ E_{2,\et}^{d-1,1}(\cRFr{d+1})=E_{2,\et}^{d-1,1}(\DRWlog{m}{d+1}{X}) = H_2(C(X_s^{\bullet},\Zpm));   \]
	\[ E_{2,\et}^{d,1}(\cRFr{d+1})=E_{2,\et}^{d,1}(\DRWlog{m}{d+1}{X}) = H_1(C(X_s^{\bullet},\Zpm)).   \]
\end{proof}

\begin{remark}
	In particular, if $X$ has good reduction, then
	\[ H^{d+1}_{X_s}(X_{\nis}, \cRFr{d+1})  \cong  H^{d+1}_{X_s}(X_{\et}, \cRFr{d+1}). \]
\end{remark}

The $p$-primary part of class field theory in this setting can be reformulated as follows:
\begin{theorem}\label{thm.cftlocalmodp}
	There is a canonical map
	\[ \rho_{X,D}\colon C(X,D;X_s)/p^m \to \pi^{\mathrm{ab}}_1(X,D)/p^m, \]
	which fits into an exact sequence of finite groups
	\[ H_2(C(X_s^{\bullet},\Zpm)) \to C(X,D;X_s)/p^m \to  \pi^{\mathrm{ab}}_1(X,D)/p^m \to H_1(C(X_s^{\bullet},\Zpm)) \to 0.\]
	In particular, we have
	\[ H_2(C(X_s^{\bullet},\Zpm)) \to \varprojlim_D (C(X,D;X_s)/p^m) \to  \pi^{\mathrm{ab}}_1(U)/p^m \to H_1(C(X_s^{\bullet},\Zpm)) \to 0.\]
\end{theorem}
\begin{proof}
	The map is defined by the following composition:
	\[ \xymatrix{ C(X,D;X_s)/p^m \ar[r]^-{\cong} \ar@{-->}[dr]_{\rho_{X,D}}& H^{d+1}_{X_s}(X_{\nis}, \mathcal{K}^M_{d+1, X|D})/p^m \ar[r]^-{\cong} & H^{d+1}_{X_s}(X_{\nis}, \cRFr{d+1}) \ar[d]\\
		&\pi^{\mathrm{ab}}_1(X,D)/p^m & H^{d+1}_{X_s}(X_{\et}, \cRFr{d+1}) \ar[l]_-{\cong}
	}\]
	where the second isomorphism in the upper row is obtained in analogy to the proof of
	Corollary~\ref{mod.pm}. Theorem~\ref{thm.cftlocalmodp} now is a consequence of Theorem \ref{nis.to.et}, Theorem \ref{idele.cft} and Theorem \ref{semistable.ramified.duality}.
\end{proof}

\providecommand{\bysame}{\leavevmode\hbox to3em{\hrulefill}\thinspace}
\providecommand{\MR}{\relax\ifhmode\unskip\space\fi MR }
\providecommand{\MRhref}[2]{
	\href{http://www.ams.org/mathscinet-getitem?mr=#1}{#2}
}
\providecommand{\href}[2]{#2}

\end{document}